\documentclass[11pt,reqno]{amsart}

\usepackage{esint}
\usepackage{comment}
\usepackage{amsmath}
\usepackage{amsthm}
\usepackage{mathrsfs}
\usepackage{graphicx}
\usepackage{mathtools}
\usepackage{amsfonts}
\usepackage{amssymb}
\usepackage[T1]{fontenc}
\usepackage{hyperref}
\usepackage[margin=1 in]{geometry}
\usepackage{enumerate}
\usepackage{cite}
\usepackage[normalem]{ulem}
\usepackage{tikz}
\usepackage{tikz-cd}
\usepackage{xcolor}
\usetikzlibrary{matrix,arrows,positioning,automata}
  \usepackage{cancel}
\usepackage{dsfont}

  \usepackage{scalerel,stackengine}
\stackMath
\newcommand\reallywidehat[1]{%
\savestack{\tmpbox}{\stretchto{%
  \scaleto{%
    \scalerel*[\widthof{\ensuremath{#1}}]{\kern-.6pt\bigwedge\kern-.6pt}%
    {\rule[-\textheight/2]{1ex}{\textheight}}
  }{\textheight}%
}{0.5ex}}%
\stackon[1pt]{#1}{\tmpbox}%
}

\newcommand{\N}{\mathbb N}
\newcommand{\Z}{\mathbb Z}

\newcommand{\R}{\mathbb R}
\def\E{\mathbb E}

\newcommand{\linf}{L^{\infty}}

\newcommand{\bP}{\mathbb{P}}

\newcommand{\bx}{\bm{x}}
\newcommand{\bz}{\bm{z}}
\newcommand{\bX}{\boldsymbol X}
\newcommand{\balpha}{\mathbf{\alpha}}

\DeclareMathOperator*{\esssup}{ess\,sup}

\newcommand{\reg}{\text{reg}}

\def\XXint#1#2#3{{\setbox0=\hbox{$#1{#2#3}{\int}$}
\vcenter{\hbox{$#2#3$}}\kern-.5\wd0}}

\newcommand{\T}{\mathbb{T}}

\newtheorem{thm}{Theorem}[section]
\newtheorem{lem}[thm]{Lemma}
\newtheorem{cor}[thm]{Corollary}
\newtheorem{prop}[thm]{Proposition}
\newtheorem{assumption}[thm]{Assumption}
\theoremstyle{definition}
\newtheorem{defn}[thm]{Definition}
\newtheorem{rmk}[thm]{Remark}

\numberwithin{equation}{section}

\def\smallnegint{\mathop{\int\mkern-13mu
        \raise.5ex\hbox{${\scriptscriptstyle\diagup}$}}\nolimits}
\def\ds{\displaystyle}

\def\tr{\operatorname{tr}}

\def\bx{{\boldsymbol x}}
\def\by{{\boldsymbol y}}
\def\ssetminus{\,\raise.4ex\hbox{$\scriptstyle\setminus$}\,}

\newcommand{\be}{\begin{equation}}
\newcommand{\ee}{\end{equation}}

\newcommand{\bc}{\begin{case}}
\newcommand{\ec}{\end{cases}}
\newcommand{\bs}{\begin{split}}
\newcommand{\es}{\end{split}}

\newcommand{\norm}[1]{\left\Vert#1\right\Vert}

\newcommand{\bm}[1]{\boldsymbol #1}

\renewcommand{\d}{d}

\newcommand{\Rd}{{\mathbb{R}^\d}}

\renewcommand{\bar}{\overline}
\renewcommand{\tilde}{\widetilde}
\renewcommand{\hat}{\widehat}

\def \be {\begin{equation}}
\def \ee {\end{equation}}

\def \E {\mathbb{E}}

\def \R {\mathbb{R}}

\def\red{\textcolor{red}}

\renewcommand{\tilde}{\widetilde}

\newcommand{\sym}{\mathrm{Sym}}
\def\red{\textcolor{red}}

\newcommand{\trip}[1]{{\left\vert\kern-0.25ex\left\vert\kern-0.25ex\left\vert #1 
    \right\vert\kern-0.25ex\right\vert\kern-0.25ex\right\vert}}

\newcommand{\cP}{\mathcal{P}}

\newcommand{\ov}{\overline}
\newcommand{\op}{\text{op}}
\newcommand{\bd}{\mathbf{d}}
\newcommand{\cA}{\mathcal{A}}
\newcommand{\cB}{\mathcal{B}}

\newcommand{\cH}{\mathcal{H}}
\newcommand{\cF}{\mathcal{F}}
\newcommand{\cL}{\mathcal{L}}

\newcommand{\EE}{{\mathbb E}}
\newcommand{\NN}{{\mathbb N}}
\newcommand{\PP}{{\mathbb P}}
\newcommand{\QQ}{{\mathbb Q}}
\newcommand{\RR}{{\mathbb R}}

\newcommand{\cadlag}{}
\def\cadlag/{c\`adl\`ag}

\newcommand{\eps}{\varepsilon}

\newcommand{\Id}{\on{Id}}

\newcommand{\oline}[1]{\overline{#1}}

\newcommand{\wt}{\widetilde}

\DeclareMathOperator*{\essinf}{ess\,inf}

\newcommand{\mcl}{\mathcal}

\newcommand{\mbf}{\mathbf}
\newcommand{\on}{\operatorname}
\newcommand{\com}{\Delta_{\text{com}}}

\begin{document}

\title[Error estimates for HJB equations on the Wasserstein space]{Error estimates for finite-dimensional approximations of Hamilton-Jacobi-Bellman equations on the Wasserstein space}

\author[S. Daudin]{Samuel Daudin}
\address{(S. Daudin) Universit\'e Paris Cit\'e, CNRS, Sorbonne Universit\'e, Laboratoire Jacques-Louis Lions (LJLL), F-75006, Paris, France
}\email{samuel.daudinATu-paris.fr
}

\author[J. Jackson]{Joe Jackson\address{(J. Jackson) The University of Chicago, Eckhart Hall, 5734 S University Ave, Chicago, IL 60637, USA}
\email{jsjackson@uchicago.edu}}
 
\author[B. Seeger]{ Benjamin Seeger
\address{(B. Seeger) University of North Carolina at Chapel Hill, 318 Hanes Hall, CB \#3260
Chapel Hill, NC 27599-3260
}\email{bseeger@unc.edu
}}

\thanks{J. Jackson was supported by the NSF under Grant No. DMS2302703. B. Seeger was partially supported by the National Science Foundation (NSF) under award number DMS-2437066.}

\begin{abstract}
    In this paper, we study a Hamilton-Jacobi-Bellman (HJB) equation set on the Wasserstein space $\cP_2(\R^d)$, with a second order term arising from a purely common noise. We do not assume that the Hamiltonian is convex in the momentum variable, which means that we cannot rely on representation formulas coming from mean field control. In this setting, Gangbo, Mayorga, and \'{S}wi\k{e}ch \cite{GMS} showed via viscosity solutions methods that the HJB equation on $\cP_2(\R^d)$ can be approximated by a sequence of finite-dimensional HJB equations. Our main contribution is to quantify this convergence result. The proof involves a doubling of variables argument, which leverages the Hilbertian approach of P.L. Lions for HJB equations in the Wasserstein space, rather than working with smooth metrics on $\cP_2(\R^d)$ which have been used to obtain similar results in the presence of idiosyncratic noise. In dimension one, our doubling of variables argument is made relatively simply by the rigid structure of one-dimensional optimal transport, but in higher dimensions the argument is significantly more complicated, and relies on some estimates concerning the "simultaneous quantization" of probability measures.
\end{abstract}

\maketitle

\setcounter{tocdepth}{1}
\tableofcontents

\section{Introduction}

This paper is concerned with the Hamilton-Jacobi-Bellman (HJB) equation 
\begin{align} \label{hjb.inf} \tag{$\text{HJB}_{\infty}$}
   \begin{cases} 
   \ds - \partial_t U(t,m) - \kappa \com U(t,m) + \int_{\R^d} H\big(x, D_mU(t,m,x),m \big) m(dx) = 0, \quad (t,m) \in [0,T) \times \cP_2, \vspace{.2cm}
   \\
   \ds  U(T,m) = G(m), \quad m \in \cP_2,
   \end{cases}
\end{align}
where $\cP_2 = \cP_2(\R^d)$ is the Wasserstein space of probability measures on $\R^d$ with finite second moment.
The data for this problem consists of 
\begin{align*}
    H(x,p,m) : \R^d \times \R^d \times \cP_2(\R^d) \to \R, \quad G(m) : \cP_2(\R^d) \to \R, \quad \kappa \geq 0.
\end{align*}
We refer to $H$ as the Hamiltonian, $G$ as the terminal condition, and $\kappa$ as the intensity of the common noise. While we do not make any precise assumptions on $H$ and $G$ at this stage, we emphasize from the start that in this paper we will \textit{not} require that $p \mapsto H(x,p,m)$ is convex, i.e. there may not be a representation formula for the solution in terms of mean field control (MFC). The unknown in \eqref{hjb.inf} is a map $U : [0,T] \times \cP_2 \to \R$. The derivative $D_m U$ appearing in \eqref{hjb.inf} is the $L$-derivative or Wasserstein derivative (see e.g. \cite[Section 5]{CarmonaDelarue_book_I}), and $\com U$ is the "common-noise operator" or "partial Laplacian" defined by 
\begin{align}
   \com U (t,m) = \int_{\R^d} \tr\big(D_{xm} U(t,m,x) \big) m(dx) + \int_{\R^d} \int_{\R^d} \tr\big(D_{mm} U(t,m,x,y) \big) m(dx) m(dy).
\end{align}

The equation \eqref{hjb.inf} formally arises as the limit as $N \to \infty$ of the following sequence of finite-dimensional HJB equations:
\begin{align} \label{hjb.n} \tag{$\text{HJB}_{N}$}
    \begin{cases}
        \ds - \partial_t V^N - \kappa \sum_{i,j = 1}^N \tr\big(D_{x^ix^j} V^N \big) + \frac{1}{N} \sum_{i = 1}^N H\big(x^i, ND_{x^i} V^N, m_{\bx}^N\big) = 0, \quad (t,\bx) \in [0,T) \times (\R^d)^N \vspace{.2cm}
        \\
       \ds  V^N(T,\bx) = G(m_{\bx}^N).
    \end{cases}
\end{align}
Here the unknown is a map $V^N : [0,T] \times (\R^d)^N \to \R$, and we use the notation $m^N_\bx = \frac{1}{N} \sum_{i=1}^N \delta_{x^i}$ for the empirical measure associated to $\bx \in (\R^d)^N$. In this paper, we will follow \cite{GMS} and interpret \eqref{hjb.inf} in the "lifted" sense of P.L. Lions, while \eqref{hjb.n} will be interpreted in the standard viscosity sense.

From a PDE perspective, the connection between \eqref{hjb.inf} and \eqref{hjb.n} can be seen from the fact that, if $U$ is smooth, then the finite-dimensional projections $U^N : [0,T] \times (\R^d)^N \to \R$, $U^N(t,\bx) = U(t,m_{\bx}^N)$ satisfy, see e.g. \cite[Proposition 6.1.1]{cdll},
\begin{align*}
    D_{x^i} U^N(t,\bx) = \frac{1}{N} D_m U(t,m_{\bx}^N), \quad D_{x^ix^j} U^N(t,\bx) = \frac{1}{N} D_{xm} U(t,m_{\bx}^N,x^i) 1_{i = j} + \frac{1}{N^2} D_{mm} U(t,m_{\bx}^N, x^i,x^j).
\end{align*}
From this, one can check that if $U$ is smooth, then $U^N$ solves \eqref{hjb.n}, and so
\begin{equation}\label{exact.projection}
    U(t, m^N_\bx) = V^N(t,\bx) \quad \text{for all } (t,\bx) \in [0,T] \times (\Rd)^N.
\end{equation}
The recent work of \'{S}wie\c{c}h and Wessels \cite{SW} shows, roughly speaking, that \eqref{exact.projection} holds as soon as $U$ is $C^1$.

Interestingly, despite the fact that \eqref{hjb.inf} formally projects exactly onto \eqref{hjb.n}, the exact projection property \eqref{exact.projection} fails in general. As an example (see a similar set-up in \cite[Proposition 6.8]{cdjm}), consider the equation
\begin{equation}\label{hjb.example}
    -\partial_t U + \frac{1}{2} \int_\Rd |D_m U(t,m,x)|^2 m(dx) = 0 \quad \text{in } [0,1) \times \mcl P_2,  \quad U(1,m) = \mbf d_2(\mu,\text{Leb}_Q),
\end{equation}
where $\mbf d_2$ denotes the $2$-Wasserstein distance and $\text{Leb}_Q \in \mcl P_2$ is the uniform measure over the cube $Q = [0,1]^d$, as well as the finite-dimensional counterpart
\begin{equation}\label{hjb.n.example}
    -\partial_t V^N + \frac{N}{2} \sum_{i=1}^N |D_{x^i} V^N|^2 = 0 \quad \text{in } [0,1) \times (\Rd)^N, \quad V^N(1,\bx) = \mbf d_2(m^N_\bx, \text{Leb}_Q).
\end{equation}
Explicit expressions for $U$ and $V^N$ coming from the connection to mean field control (more details on this below) yield
\[
    U(0, m^N_\bx) = \inf_{\nu \in \mcl P_2} \left\{ \mbf d_2(\nu, \text{Leb}_Q) + \frac{1}{2} \mbf d_2^2(m^N_\bx, \nu) \right\}
\]
and
\[
    V^N(0, \bx) = \inf_{\by \in (\Rd)^N} \left\{ \mbf d_2(m^N_\by, \text{Leb}_Q) + \frac{1}{2} \mbf d_2^2(m^N_\bx, m^N_\by) \right\},
\]
from which it is easy to see that
\[
    \sup_{\bx \in (\Rd)^N} \left( V^N(0, \bx) - U(0, m^N_\bx) \right) \ge r_N - \frac{1}{2} r_N^2, \quad 
    r_N = \inf_{\bx \in (\Rd)^N} \mbf d_2(m^N_\bx, \text{Leb}_Q ) > 0.
\]
The interpretation of this discrepancy is that, for MFC problems corresponding to these equations, even for populations starting in the discrete configuration $m^N_\bx$, it is beneficial to diffuse over time so as to best approach the uniform measure $\text{Leb}_Q$ - although the value of this benefit converges to $0$ as $N \to \infty$ at precisely the rate $r_N$.

While the example above suggests that $V^N$ should converge to $U$ at a rate related to certain "quantization errors" like $r_N$, so far only qualitative convergence is known in this setting. In particular, in \cite{GMS}, it was shown that the viscosity solutions $V^N$ of \eqref{hjb.n} converge locally uniformly towards the unique lifted viscosity solution $U$ of \eqref{hjb.inf}, in the sense that for each bounded subset $\mathcal{B}$ of $\cP_2$, we have
\begin{align}
    \lim_{N \to \infty} \sup_{(t,m_{\bx}^N) \in \mathcal{B}} \big| U(t,m_{\bx}^N) - V^N(t,\bx) \big| = 0. \label{convergence}
\end{align}
The goal of this paper is to quantify this convergence.



\subsection{Related literature} In order to explain where our results fit into the literature, we first discuss some variations of \eqref{hjb.inf} which have also been studied in recent years. First, we use the term "convex case" to refer to the case that $p \mapsto H(x,p,m)$ is convex; this should not be confused with convexity in $x$ and/or $m$. In the convex case (at least under standard technical assumptions), we can write $H$ as
\begin{align*}
    H(x,p,m) = \sup_{a \in \R^d} \Big\{ - L(x,a,m) - a \cdot p \Big\}
\end{align*}
for some $L : \R^d \times \R^d \times \cP_2 \to \R$. Then $V^N$ admits the representation formula 
\begin{align*}
    V^N(t_0,\bx_0) = \inf_{\balpha = (\alpha^1,\ldots,\alpha^N)} \E\bigg[ \int_{t_0}^T \frac{1}{N} \sum_{i = 1}^N L\Big(X_t^i, \alpha_t^i , m_{\bX_t}^N \Big) dt + G(m_{\bX_T}^N)\bigg], 
\end{align*}
where $\bX = (X^1,\ldots,X^N)$ is determined from $\balpha = (\alpha^1,\ldots,\alpha^N)$ by the dynamics 
\begin{align*}
    dX_t^i = \alpha^i_t + \sqrt{2 \kappa} dB_t, \quad t_0 \leq t \leq T, \quad X_{t_0}^i = x_0^i, 
\end{align*}
$B$ is a standard Brownian motion, and the infimum is taken over all square-integrable, $(\R^d)^N$-valued processes $\balpha = (\alpha^1,\ldots,\alpha^N)$ progressively measurable with respect to $B$. Similarly, $U$ is the value function of a mean field control (MFC) problem which takes the form 
\begin{align*}
    U(t_0,m_0) = \inf_{\alpha} \E\bigg[ \int_{t_0}^T L\Big(X_t, \alpha_t, \cL(X_t | B) \Big) dt + G(X_T) \bigg], 
\end{align*}
where $X$ satisfies 
\begin{align*}
    dX_t = \alpha_t + \sqrt{2 \kappa} dB_t, \quad X_{t_0} = \xi \sim m_0, \quad X_{t_0} \perp B. 
\end{align*}
This connection to MFC makes the convex case significantly different from the non-convex case, where one must rely on purely PDE techniques to study \eqref{hjb.inf} and its connection to \eqref{hjb.n}. 

In addition to distinguishing between the convex and nonconvex settings, we also point out that different noise structures are possible. In particular, the operator $\com$ represents a common noise; in the convex case, this means that each of the particles $X^i$ is impacted by the same Brownian motion $B$. It is also typical to consider problems with a non-degenerate idiosyncratic noise. In terms of \eqref{hjb.inf}, this would mean that $\kappa \com$ is replaced by $\Delta_{\text{id}} + \kappa \com$, where $\Delta_{\text{id}}$ is the idiosyncratic noise operator
\begin{align*}
    \Delta_{\text{id}} U(t,m) = \int_{\R^d} \tr \big( D_{xm} U(t,m,x) \big) m(dx). 
\end{align*}
In the convex case, this would mean that the term $\sqrt{2 \kappa} dB$ appearing in the dynamics of $X^i$ would be replaced by $\sqrt{2} dW^i + \sqrt{2\kappa} dB$, where $(W^i)_{i = 1,\ldots,N}$ are independent Brownian motions independent of $B$. 

Finally, it is also worth pointing out that many previous papers work in a periodic setting, by replacing $\R^d$ everywhere with the $d$-dimensional flat torus $\T^d = \R^d / \mathbb{Z}^d$. We refer to this as the periodic or compact case, as opposed to the non-compact case treated in the present paper.

We have thus outlined several axes along which HJB equations on the Wasserstein space can be placed: convex vs. non-convex Hamiltonian, non-degereate vs. degenerate idiosyncratic noise, with common noise vs. without commmon noise, and compact/periodic vs. noncompact. With this terminology in hand, we will now discuss the existing literature related to the well-posedness of \eqref{hjb.inf} and the convergence of $V^N$ to $U$.
\newline \newline 
\noindent 
\textit{HJB equations on spaces of measures.} Equations of the form \eqref{hjb.inf} have received significant attention in recent years, in large part because of their connection to mean field control. A large body of literature has been devoted to developing an appropriate notion of viscosity sub/supersolution and to establishing a comparison principle. We refer to \cite{GT_19, soneryan2024, soneryan2024B, BayraktarEkrenZhangCPDE, BayraktarEkrenHeZhang, touzizhangzhou, bertucci2023, confortihj, daudinseeger, DJS} for a sample of recent efforts in this direction. One of the reasons for the huge literature on the topic is that, unlike in finite dimensions, there are several reasonable definitions of viscosity solutions for equations set on $\cP_2$. Moreover, both the definition which is most convenient and the strategy of proof depends on the factors discussed above; for example the arguments in \cite{cossogozzikharroubi, soneryan2024B, daudinseeger, DJS, BayraktarEkrenZhangCPDE} are tailored to equations with non-degenerate idiosyncratic noise, while \cite{GT_19, bertucci2023} focus on problems with a purely common noise (as in the present paper), and \cite{cossogozzikharroubi} uses the connection to MFC available in the convex case. Most relevant to the present paper are \cite{GT_19, GMS}, where equations with purely common noise (as in \eqref{hjb.inf}) are treated via the "lifting" approach of P.L. Lions. This method will be discussed in detail below; we emphasize for now that it so far has proved very efficient for problems with purely common noise, but has not been adapted to treat problems with idiosyncratic noise. 
\newline \newline 
\noindent 
\textit{Qualitative convergence.}  In the convex case, the statement \eqref{convergence} is intimately related to the convergence problem in mean field control, and the convergence of $V^N$ to $U$ can be approached by probabilistic compactness arguments. This approach was first used by Budhiraja, Dupuis, and Fischer \cite{BDF} in the context of large deviations for interacting particle systems, and later extended to general MFC problems by Lacker \cite{Lacker} and Djete, Possamai, and Tan \cite{DPT}, the latter incorporating both idiosyncratic and common noise. Meanwhile, as discussed above \cite{GMS} establishes \eqref{convergence} for the exact model treated here (non-convex, with common noise but without idiosyncratic noise) via a viscosity solutions argument: first they show that the sequence $(V^N)_{N \in \N}$ is equicontinuous in an appropriate sense, then they show that all of its limit points solve \eqref{hjb.inf} in the "lifted" viscosity sense, and finally they conclude \eqref{convergence} from the uniqueness of viscosity solutions. A similar PDE argument is executed in \cite{DJS} for non-convex problems with both idiosyncratic and common noise. 
\newline \newline 
\noindent 
\textit{Quantitative convergence.} The first general quantitative convergence result came in \cite{CDJS}, which established a bound of the form 
\begin{align} \label{cdjsrate}
    |V^N(t,\bx) - U(t,m_{\bx}^N)| \leq C \Big(1 + \frac{1}{N} \sum_{i = 1}^N |x^i|^2 \Big) N^{-\gamma_d}, 
\end{align}
with $\gamma_d \in (0,1)$ a non-explicit dimensional constant. This result was obtained in the convex case, with non-degenerate idiosyncratic noise and common noise. Since then, there have been many efforts to sharpen and generalize the result \eqref{cdjsrate}, and in particular understand the optimal exponent $\gamma_d$. We mention in particular \cite{ddj2023, cjms2023, cdjm, BEZ} for developments in this direction. For the present discussion, the important thing to point out is that the existing quantitative convergence results all require convexity of $p \mapsto H(x,p,m)$ and/or non-degenerate idiosyncratic noise. In particular, none apply to the equation \eqref{hjb.inf}.

\subsection{Our methods} \label{subsec.methods} As discussed above, the existing quantitative convergence results all require either a convex (in $p$) Hamiltonian or non-degenerate idiosyncratic noise. This is because the arguments in \cite{CDJS, ddj2023, cdjm, BCC} either use explicitly control-theoretic arguments, or they make use of certain "smooth metrics" on $\cP(\T^d)$, like the one obtained by embedding $\cP(\T^d)$ into the (negative index) Sobolev space $H^{-s}$, $s > \frac{d}{2} + 2$. The arguments based on smooth metrics in \cite{ddj2023, cdjm, BCC} do not extend easily to the case of degenerate idiosyncratic noise; this is linked to the fact that one cannot expect Lipschitz regularity with respect to $H^{-s}$ without nondegenerate idiosyncratic noise, even when all the data is smooth, as discussed in \cite{cdjm}. 

On the other hand, for problems with a purely common noise (as in \eqref{hjb.inf}), the "lifting" approach used in \cite{GT_19 ,GMS} has proved very efficient. Thus, it seems natural to try to "quantify" the lifting-based arguments from \cite{GT_19, GMS} in order to obtain quantitative convergence in the case with purely common noise. This is the approach we take in the present paper. In order to explain our strategy in detail, consider a simplified, stationary version of the equations for $V^N$ and $U$, without common noise:
\begin{align} \label{hjb.inf.stat}
    U(m) + \int_{\R^d} H\big(D_m U(m,x)\big) m(dx) = F(m), \quad m \in \cP_2, 
\end{align}
and 
\begin{align} \label{hjb.n.stat}
    V^N(\bx) + \frac{1}{N} \sum_{i = 1}^N H\big( ND_{x^i} V^N(\bx) \big) = F(m_{\bx}^N), \quad \bx \in (\R^d)^N,  
\end{align}
where $H : \R^d \to \R$ and $F : \cP_2 \to \R$ are both Lipschitz for simplicity. 
Let $(\Omega, \cF, \bP)$ be a non-atomic probability space, and $\cH = L^2(\Omega; \R^d)$. 
Roughly speaking, using the lifting approach means that we define $U$ to be a viscosity solution of \eqref{hjb.inf.stat} if its "lift" 
\begin{align*}
    \hat{U} : \cH \to \R, \quad \hat{U}(X) = U\big(\cL(X) \big)
\end{align*}
satisfies 
\begin{align} \label{hjb.hilbert.stat}
    \hat{U}(X) + \hat{H}\big( \nabla \hat{U}(X) \big) = F\big(\cL(X)\big), \quad X \in \cH, 
\end{align}
where $\hat{H}(P) = \E\big[ H(P) \big]$. A natural way to attempt to quantify the arguments in \cite{GMS} is as follows. Ignore compactness issues, and suppose we can find a maximum point $(\ov X, \ov \bx)$ of 
\begin{align*}
  \cH \times (\R^d)^N \ni (X, \bx) \mapsto \hat{U}(X) - V^N(\bx) - \frac{1}{2\eps} \| X - Y_{\bx}^N\|_2, 
\end{align*}
where $Y_{\bx}^N = \sum_{i = 1}^N x^i 1_{\Omega^i}$, and $\Pi_N = \{\Omega^1,\ldots,\Omega^N\}$ is a partition of $\Omega$ into $N$ sets of equal probability. Then, applying the viscosity subsolution test $\hat{U}$ and the super-solution test for $V^N$, and computing the derivatives of $\bx \mapsto \frac{1}{2\eps} \| X - Y_{\bx}^N\|_2^2$ via Lemma \ref{lem.projection.deriv} below, we arrive at 
\begin{align*}
    \sup_{\bx} &\Big \{ U(m_{\bx}^N) - V^N(\bx) \Big\} \leq \hat{U}(\ov X) - V^N(\ov \bx)
    \\
    &\leq \hat{H}\Big( \E\Big[ \frac{\ov X - Y_{\ov \bx}^N}{\eps} | \cF_{\Pi_N} \Big] \Big) - \hat{H}\Big( \frac{\ov X - Y_{\ov \bx}^N}{\eps} \Big) + F(\cL(\ov{X})) - F(m_{\ov \bx}^N)
    \\
    &\lesssim \eps^{-1} \| \ov{X} - \E[\ov X | \cF_{\Pi_N}] \|_1 + \|\ov{X} - Y_{\ov \bx}^N\|_2
    \\
    &\lesssim \eps^{-1} \| \ov{X} - \E[\ov X | \cF_{\Pi_N}] \|_1 + \eps, 
\end{align*}
where $\cF_{\Pi_N}$ is the discrete $\sigma$-algebra generated by $\Pi_N$, and the last bound holds if $U$ is Lipschitz with respect to $\bd_2$, by a standard argument. The question is thus whether it is possible to choose $\eps = \eps_N$ going to zero in such a way that $\eps_N^{-1} \| \ov{X} - \E[\ov X | \cF_{\Pi_N}] \|_1$ vanishes with $N$. The best one can hope for is that $\|\ov{X} - \E[\ov X | \cF_{\Pi_N}] \|_1$ might be close to the "quantization error" 
\begin{align*}
    \inf_{\Pi \in P_N} \norm{ \ov X - \E\big[ X | \cF_{\Pi} \big] }_{1}, 
\end{align*}
where the infimum is taken over the set $P_N$ of partitions of $\Omega$ into $N$ sets of equal measure. It is known that this quantization error is of order $N^{-1/d}$, at least when the $\linf$ norm of $\ov X$ is controlled (a point which we ignore in this discussion). Plugging this in above and choosing $\eps_N = N^{-1/(2d)}$, we see that this argument could at best lead to the rate $N^{-1/(2d)}$, if we are lucky and $ \norm{ \ov X - \E\big[ \ov X | \cF_{\Pi} \big] }_{1}$ is comparable to $ \inf_{\Pi \in P_N} \norm{ X - \E\big[ X | \cF_{\Pi} \big] }_{1}$. We use the term "quantization error" here because of the connection to the optimal quantization of probability measures, i.e. the problem of approximating a general probability measure $m$ by a measure supported on at most $N$ points, as surveyed in \cite{DSS}. The relevant quantization problem here is slightly different than the one studied in \cite{DSS}, since we need to approximate $m$ by an empirical measure (in which each atom has equal mass); more on this distinction in Section \ref{sec.simquant} below.

It turns out that the argument outlined above works (after updates to account for the time-dependence, the common noise, and the lack of compactness) only when $d = 1$, because in this case the rigid structure of optimal transport in dimension $1$ guarantees that we can find a maximum point $(\ov X, \ov Y)$ such that  $\|\ov X - \E\big[ X | \cF_{\Pi_N} ]\|_{\infty}$ vanishes with $N$. When $d \geq 2$, this argument breaks down completely, and it is not clear to us how to obtain a useful bound on $\eps^{-1} \| \ov{X} - \E[\ov X | \cF_{\Pi_N}] \|_1$. 

To circumvent this issue, we start by extending $V^N$ in a Lipschitz way to a map $\wt{V}^N : \cP_2 \to \R$, i.e. we find $\wt V^N$ such that $\wt{V}^N(m_{\bx}^N) = V^N(\bx)$, and then we define $\hat{V} : \cH \to \R$ by $\hat{V}^N(X) = \wt{V}^N(\cL(X))$. We then look for a maximum point $(\ov X, \ov Y)$ (again ignoring compactness issues for now) of 
\begin{align} \label{doubling.intro}
  \cH \times \cH \ni (X,Y) \mapsto \hat{U}(X) - \hat{V}^N(Y) - \frac{1}{2\eps} \| X - Y\|_2^2. 
\end{align}
Of course, $\hat{V}^N$ does not satisfy an equation, so we cannot obtain an upper bound directly from this doubling of variables; instead  we must then consider a second optimization problem. We fix a partition $\Pi_N = \{\Omega^1,\ldots,\Omega^N\}$ as above, and look for a minimum of 
\begin{align*}
  \bx \mapsto V^N(\bx) + \frac{1}{2\eps} \| \ov X - Y_{\bx}^N \|_2 + \delta h\big( \| Y_{\bx}^N - \ov Y\|_2\big), 
\end{align*}
where we again set $Y_{\bx}^N = \sum_{i = 1}^N x^i 1_{\Omega^i}$, and
with $h$ an appropriate penalization function which forces the maximum point $\ov \bx$ to satisfy $Y_{\ov \bx}^N \approx \ov Y$. Then we apply the subsolution test to $\hat{U}$ at $\ov X$, and the supersolution test for $V^N$ at $\ov \bx$, and try to estimate the errors; the key point here is that because $\Pi_N$ does not appear in \eqref{doubling.intro}, it can be chosen as a parameter after $\ov X$ and $\ov Y$ have been fixed, and this helps to control the relevant errors, which turn out to involve terms like 
\begin{align} \label{sim.quant}
    \norm{ \E\Big[\frac{ \ov X - \ov Y}{\eps} | \cF_{\Pi_N} \Big] - \frac{ \ov X - \ov Y}{\eps} }_1, \quad \norm{\ov Y - \E \big[ \ov Y | \cF_{\Pi_N} \big]}_2, 
\end{align}
along with other terms involving $\delta$, $\eps$, and other parameters. In particular, it becomes important to choose $\Pi_N$ in such a way that both of the terms in \eqref{sim.quant} are small, which leads us to a "simultaneous quantization" problem discussed in Section \ref{sec.simquant}. In the end, we succeed in choosing the various parameters, including the partition $\Pi_N$, as functions of $N$ in order to obtain an upper bound on $V^N - U$. 

The actual argument is fairly involved, because we need to deal with the time dependence, the compactness, and the common noise, with the common noise in particular adding considerable complexity. The argument above shows that the best rate we can hope to obtain from this argument (at least for a merely Lipschitz Hamiltonian) is $N^{-1/(2d)}$. Because of the various additional technicalities faced in the argument, we end up with something slightly worse; in the case when $d \geq 3$ and $\kappa = 0$, for instance, we get a rate of $N^{-1/(5d)}$. 

\subsection{Assumptions and main results}

The basic assumption under which we prove error estimates reads as follows:

\begin{assumption} \label{assump.main}
    There are constants $C_H, C_G > 0$ such that
    \begin{align}
       &|H(x,p,m)|\leq C_H \big(1 + |p|^2 \big), \label{ass:quadraticgrowthH}
       \\
       &|H(x,p,m) - H(x',p',m')| \leq C_H \big(1 + |p| \big)\Big(|x -x'| + |p-p'| + \bd_1(m,m') \Big), \label{normalHreg}
    \end{align}
    for all $x,x',p,p' \in \R^d$ and $m,m' \in \cP_2(\R^d)$, and
    \begin{align}
       |G(m)| \leq C_G, \quad  |G(m) - G(m')| \leq C_G \bd_1(m,m')
    \end{align}
    for each $m,m' \in \cP_2(\R^d)$. 
\end{assumption}

\begin{rmk}
    The assumptions on the growth of $H$ and $D_p H$ are exactly the same as those in \cite[Hypothesis 2.1]{GMS}. For the continuity of $H$ and $G$ in $(x,m) \in \Rd \times \cP_2$, we consider the special case of Lipschitz regularity with respect to $\mbf d_1$. Error estimates for other moduli of continuity and for other Wasserstein metrics can be established through a careful adaptation of the methods in this work; in order to keep the presentation as clean as possible, we focus on Assumption \ref{assump.main}.
\end{rmk}

Our main result is as follows:

\begin{thm} \label{thm.main}
    Let Assumption \ref{assump.main} hold. Then there is a constant $C$ such that for all $N \in \N$, and all $(t,\bx) \in [0,T] \times (\R^d)^N$, we have 
    \begin{align} \label{rates.common}
       | V^N(t,\bx) - U(t,m_{\bx}^N) | \leq C \Big(1 + \frac{1}{N} \sum_{i = 1}^N |x^i|^2 \Big) \times \begin{cases}
           N^{-\frac{1}{3}} & d = 1, 
           \\
           N^{- \frac{1}{14}} \log(1 + N)^{1/7} & d = 2
           \\
           N^{- \frac{1}{7d}} & d \geq 3.
       \end{cases}
    \end{align}
    Moreover, if $\kappa = 0$, then the rates can be improved in dimension $d \geq 2$ as follows:
    \begin{align} \label{rates.nocommon}
       | V^N(t,\bx) - U(t,m_{\bx}^N) | \leq C \Big(1 + \frac{1}{N} \sum_{i = 1}^N |x^i|^2 \Big)  \times \begin{cases}
           N^{-\frac{1}{3}} & d = 1, 
           \\
           N^{- \frac{1}{10}} \log(1 + N)^{1/5} & d = 2
           \\
           N^{- \frac{1}{5d}} & d \geq 3.
       \end{cases}
    \end{align}
\end{thm}

\begin{rmk}
    In light of Proposition 6.8 of \cite{cdjm}, the rate cannot be faster than $N^{-1/d}$ in general. On the other hand, this example involves a purely quadratic Hamiltonian, and the heuristic argument explained in Subsection \ref{subsec.methods} suggests that (at least when $H$ is only Lipschitz), our doubling of variables argument cannot achieve a better rate than $N^{-1/(2d)}$. In any case, the rates achieved in Theorem \ref{thm.main} are almost certainly not optimal, and we discuss in Section \ref{sec.improvements} below how to improve them under additional assumptions on $H$ and $G$. The best rate we can get from our proof techniques comes when $\kappa = 0$, the data is periodic, and $H$ is $C^{1,1}$ in $p$; in this case we can obtain the rate $N^{-2/(7d)}$. 
\end{rmk}

\subsection{Organization of the paper}

In Section \ref{sec.prelim}, we introduce the relevant notation and discuss the definition of L-viscosity solution for \eqref{hjb.inf}, and also prove some relevant facts about quantization. In Section \ref{sec.main}, we present the proof of Theorem \ref{thm.main}, first in the case $d =1$ (Subsection \ref{subsec.d1}, then in the case $d \geq 2$ (Subsection \ref{subsec.main.ge2}).

\section{Preliminaries}  \label{sec.prelim}

\subsection{Notation} We fix throughout the paper the time horizon $T > 0$ and the dimension $d \ge 1$. Throughout the paper, we work for concreteness with the probability space
\[
    (\Omega, \cF, \bP) = \big([0,1], \mathcal{B}([0,1]), \text{Leb}\big),
\]
i.e. the unit interval equipped with the Borel $\sigma$-algebra and the Lebesgue measure. For most of the paper the choice of probability space plays absolutely no role, but it makes some of the arguments about quantizing one-dimensional random variables more transparent. For $1 \leq p \leq \infty$ and $k \in \N$, we denote by $L^p(\Omega; \R^k)$ the Lebesgue space of $p$-integrable, $k$-dimensional, random vectors with the usual norm 
\begin{align*}
    \| X \|_p = \E\big[ |X|^p \big]^{1/p}, \quad  1 \leq p < \infty, \quad  \| X\|_{\infty} =  \esssup |X|.
\end{align*}
We also use the notation $\cH = L^2(\Omega ; \R^d)$, which is a separable Hilbert space endowed with the inner product 
\begin{align*}
    \langle X, Y \rangle = \E\big[ X \cdot Y \big] = \int_0^1 X(\omega) \cdot Y(\omega) d\omega. 
\end{align*}

Given a Hilbert space $H$ with inner product $\langle \cdot, \cdot \rangle$ (possibly different from $\cH$), 
we denote by $\sym (H)$ the space of symmetric bilinear forms on $H$. Given $\cA \in \sym(H)$, we use the notation 
\begin{align*}
  H \times H \ni (u,v) \mapsto \cA [u,v] \in \R
\end{align*}
for the action of $\cA$ on the vectors $u$ and $v$. We write $\cA \leq \cB$ for the usual partial order on $\text{Sym}(H)$. Each $\cA \in \sym(H)$ is identified with a bounded linear operator $H \to H$ in the usual way, and we will abuse notation by using the same symbol for both objects, i.e. 
\begin{align*}
    \langle u, \cA v \rangle = \cA[u,v], \quad u,v \in \cH. 
\end{align*}
We write $\| \cA \|_{\op}$ for the operator norm of $\cA$ when viewed as a linear operator. Given a differentiable function $\Phi : H \to \R$, we denote by $\nabla \Phi$ the gradient of $\Phi$, i.e. $\nabla \Phi(X) \in H$ satisfies 
\begin{align*}
    \Phi(X + Y) = \Phi(X) + \langle \nabla \Phi(X), Y \rangle + o(\|Y\|_2). 
\end{align*}
We use $\nabla^2 \Phi$ to denote the Hessian of $\Phi$ if it exists, i.e. $\nabla^2 \Phi(X) \in \sym(H)$ satisfies 
\begin{align*}
    \langle \nabla \Phi(X + Y), Z \rangle = \langle \nabla \Phi(X), Z \rangle + \nabla^2 \Phi(X) [Y,Z] + o(\|Y\|). 
\end{align*}
We use similar notations for functions $\Phi : [0,T] \times \cH \to \R$. 

We also work extensively in this paper with the Wasserstein space $\cP_2 = \cP_2(\R^d)$, the space of Borel probability measures $m$ on $\R^d$ with 
\begin{align}
\label{defn:M2}
    M_2(m) \coloneqq \int_{\R^d} |x|^2 m(dx) < \infty. 
\end{align}
We endow $\cP_2$ with the $2$-Wasserstein distance $\bd_2$, defined by 
\begin{align*}
    \bd_2(m,n) = \inf_{\gamma \in \Pi(m,n)} \Bigl \{  \int_{\R^d \times \R^d} |x-y|^2 \gamma(dx,dy) \Bigr \}^{1/2}  = \inf_{\substack{X,Y \in \cH \\ \cL(X) = m, \, \cL(Y) = n}} \|X-Y\|_2, 
\end{align*}
where $\Pi(m,n)$ is the set of all probability measures $\gamma$ on $\R^d \times \R^d$ with $\big((x,y) \mapsto x \big)_{\#} \gamma = m$ and $\big((x,y) \mapsto y \big)_{\#} \gamma = n$, and $\cL(X) = X_{\#} \bP$ is the law of a random variable $X$. Given $\Phi : \cP_2 \to \R$, we denote by $\hat{\Phi}$ its "lift" to $\cH$, defined by 
\begin{align*}
    \hat{\Phi}(X) = \Phi\big( \cL(X) \big), \quad X \in \cH. 
\end{align*}
If $\Phi$ is smooth enough, we denote by $D_m \Phi$ its "intrinsic" or ``Wasserstein'' derivative, i.e. $D_m \Phi : \cP_2 \times \R^d \to \R^d$ satisfies 
\begin{align*}
    D_m \Phi\big( \cL(X), X \big) = \nabla \hat{\Phi}( X ), \quad X \in \cH. 
\end{align*}
We refer to \cite{CarmonaDelarue_book_I} Chapter 5 for details regarding the various notions of derivatives over the Wasserstein space.
We use similar notations for functions $\Phi : [0,T] \times \cP_2 \to \R$. 

Finally, for $N \in \N$, we will often work with the space $(\R^d)^N$. We write $\bx = (x^1,\ldots,x^N)$ for the general element of this space, with $x^i \in \R^d$, and we can expand further by writing $x^i = (x^i_1,\ldots,x^i_d)$ if necessary. For functions $\phi : (\R^d)^N \to \R$, we write $D_{x^i} \phi $ for the gradient of in the variable $x^i$, i.e. $D_{x^i} \phi(\bx) = (D_{x^i_1} \phi(\bx), \ldots, D_{x^i_d} \phi(\bx)) \in \R^d$. Similarly, for $i,j = 1,\ldots,N$, we write $D_{x^ix^j} \phi$ for the element of $\R^{d \times d}$ given by $\big(D_{x^ix^j} \phi(\bx)\big)_{q,r} =  D_{x^i_q x^j_r} \phi(\bx)$, $q,r = 1,\ldots,d$. The full Hessian $D^2 \phi$ takes values in $(\R^{d \times d})^{N \times N} \cong \R^{Nd \times Nd}$, with $(D^2 \phi)_{i,j} = (D_{x^ix^j} \phi)_{i,j = 1,\ldots,N}$.

\subsection{$L$-viscosity solutions}

The motivation for the definition of $L$-viscosity solutions comes from the fact if $U$ satisfies \eqref{hjb.inf}, then formally its "lift"
\begin{align*}
    \hat{U} : [0,T] \times \cH \to \R, \quad \hat{U}(t,X) = U\big( t, \cL(X) \big)
\end{align*}
satisfies the equation 
\begin{align} \label{hjb.hilbert} \tag{$\text{HJB}_{\cH}$}
  \begin{cases}
      \ds  - \partial_t \hat{U}(t,X) - \kappa \sum_{k = 1}^d D^2 \hat{U}(t,X) \big[e_k, e_k \big] + \hat{H}\big(X, D \hat{U}(t,X) \big) = 0, \quad (t,X) \in [0,T) \times \cH, 
     \vspace{.2cm}  \\
    \ds   \hat{U}(T,X) = \hat{G}(X), \quad X \in \cH,
  \end{cases} 
\end{align}
 where $e_1,\ldots,e_d$ are the standard unit basis vectors in $\R^d$ (which embed into $\cH$ in the obvious way, as constant random vectors) and we define $ \hat{H}(X,P) : \cH \times \cH \to \R$ and $\hat{G}(X) : \cH \to \R$ by the formulas
\begin{align*}
   \hat{H}(X,P) = \E\Big[ H\big(X,P,\cL(X)\big) \Big], \quad \hat{G}(X) = G\big( \cL(X) \big).
\end{align*}
We can define viscosity (sub/super)-solutions of \eqref{hjb.hilbert} in a standard way, see e.g. \cite{LIONS19891}.

\begin{defn} \label{defn.hilbertviscosity}
    An upper semicontinuous function $V : [0,T] \times \cH \to \R$ is a viscosity subsolution of \eqref{hjb.hilbert} if for each $\Phi \in C^{1,2}([0,T] \times \cH)$ and each $(t_0,X_0) \in [0,T) \times \cH$ such that 
    \begin{align} \label{touchingfromabove}
        V(t_0,X_0) - \Phi(t_0,X_0) = \sup_{t,X \in [0,T] \times \cH} \Big\{ V(t,X) - \Phi(t,X) \Big\}, 
    \end{align}
    we have 
    \begin{align*}
        - \partial_t \Phi(t_0,X_0) - \kappa \sum_{k = 1}^d D^2 \Phi(t_0,X_0)(e_k,e_k) + \hat{H}\big(X_0, D \Phi(t_0,X_0) \big) \leq 0.
    \end{align*}
    Likewise, a lower semicontinuous function $V : [0,T] \times \cH \to \R$ is a viscosity super-solution of \eqref{hjb.hilbert} if for each $\Phi \in C^{1,2}([0,T] \times \cH)$ and each $(t_0,X_0) \in [0,T) \times \cH$ such that 
    \begin{align} \label{touchingfrombelow}
        V(t_0,X_0) - \Phi(t_0,X_0) = \inf_{t,X \in [0,T] \times \cH} \Big\{ V(t,X) - \Phi(t,X) \Big\}, 
    \end{align}
    we have 
    \begin{align*}
        - \partial_t \Phi(t_0,X_0) - \kappa \sum_{k = 1}^d D^2 \Phi(t_0,X_0)(e_k,e_k) + \hat{H}\big(X_0, D \Phi(t_0,X_0) \big) \geq 0.
    \end{align*}
\end{defn}

\begin{rmk} \label{rmk.jets}
    It is a standard fact (see e.g. \cite{LIONS19891, CKS}) that Definition \ref{defn.hilbertviscosity} can be equivalently written in terms of the parabolic sub-jets or super-jets of the function $V$. More precisely, given $V : [0,T] \times \cH \to \R$, and $(t_0,X_0) \in [0,T) \times \cH$, we define $J^{2,+}V(t_0,X_0)$ to be the set of $(r,p,Q) \in \R \times \cH \times \text{Sym}(\cH)$ such that
    \[
        V(t,X) - V(t_0,X_0) \le r(t-t_0) + \langle p, X - X_0 \rangle + \frac{1}{2} Q[ X - X_0, X - X_0] + o\big( \norm{X - X_0}_2^2 + |t-t_0| \big)
    \]
    as $(t,X) \to (t_0,X_0)$. We also define $\ov{J}^{2,+}(t_0,X_0)$ to be the set of $(r,P,Q) \in \R \times \cH \times \text{Sym}(\cH)$ such that there exist a sequence $(t_n,X_n,r_n,p_n,Q_n)_{n = 1,\ldots,\infty} \subset [0,T) \times \cH \times \R \times \cH \times \text{Sym}(\cH)$ with $(t_n,X_n) \to (t_0,X_0)$, $(r_n,p_n,Q_n) \in J^{2,+}(t_n,X_n)$, and $(r_n,p_n,Q_n) \to (r,P,Q)$. The sets $J^{2,-}V(t_0,X_0)$ and $\ov{J}^{2,-}V(t_0,X_0)$ are defined analogously. 
    
    Then, in view of the continuity of $H$, $V$ is a viscosity subsolution of \eqref{hjb.hilbert} if and only if for each $(t_0,X_0) \in [0,T) \times \cH$ and $(r,P,Q) \in \ov{J}^{2,+}(t_0,X_0)$, we have 
    \begin{align*}
        - r - \kappa \sum_{k = 1}^d Q[e_k,e_k] + \hat{H}\big(X_0, P \big) \leq 0.
    \end{align*}
    Similarly, $V$ is a viscosity supersolution of \eqref{hjb.hilbert} if and only if for each $(t_0,X_0) \in [0,T) \times \cH$ and $(r,P,Q) \in \ov{J}^{2,-}(t_0,X_0)$, we have 
    \begin{align*}
        - r - \kappa \sum_{k = 1}^d Q[e_k,e_k] + \hat{H}\big(X_0, P \big) \geq 0.
    \end{align*}
\end{rmk}

We then follow \cite{GMS}, and define a corresponding notion of viscosity solution for \eqref{hjb.inf} as follows:

\begin{defn} \label{defn.Lviscosity}
    A function $\Phi : [0,T] \times \cP_2 \to \R$ is a viscosity sub-solution to \eqref{hjb.inf} if its lift $\hat{\Phi} : [0,T] \times \cH \to \R$, $\hat{\Phi}(t,X) = \Phi(t,\cL(X))$ is a viscosity sub-solution of \eqref{hjb.hilbert}. Super-solutions are defined analogously, and a viscosity solution is a functions which is both a sub-solution and a super-solution.
\end{defn}

It is shown in \cite{GMS} that under Assumption \ref{assump.main} (in fact, under weaker conditions), \eqref{hjb.inf} has a unique viscosity solution in the sense of Definition \ref{defn.Lviscosity}.

We finally introduce a slight adaptation of \cite[Theorem 3.2]{CKS}, and which will be useful in the doubling of variables arguments below. First, given a Hilbert space $H$ and a function $\Phi : [0,T) \times H \to \R$, we consider the following condition:
\begin{equation}\label{BT}
    \left\{
    \begin{split}
   &\text{for each $M > 0$, there exists $C = C(M) > 0$ such that } \\
    &\text{if $(t,u) \in [0,T) \times H$, $(a, p, \cA) \in J^{2,+} \Phi (t,u)$, and $\| p\|_H + \|\cA\|_{\op} \leq M$, then} \\
    &-a \leq  C.
    \end{split}
    \right.
    \tag{BT}
\end{equation}
Fix two Hilbert spaces $H_1$ and $H_2$, and let $Z_1$ and $Z_2$ be finite-dimensional subspaces of $H_1$ and $H_2$ respectively. We denote by $W_i$ the orthogonal complement of $Z_i$ in $H_i$, $i = 1,2$,  $P_i : H_i \to H_i$ the orthogonal projection onto $Z_i$, and $Q_i = \text{Id}_{H^i} - P_i$ the orthogonal projection onto $W_i$.

\begin{lem} \label{lem.cks}
    Suppose that $U_1 : (0,T) \times H_1 \to \R$ and $U_2 : (0,T) \times H_2 \to \R$ are both continuous and satisfy the condition \eqref{BT}, and define $U : (0,T) \times H_1 \times H_2 \to \R$ by $U(t,X_1,X_2) = U_1(t,X_1) + U_2(t,X_2)$. Suppose further that for some $(\ov t,\ov X_1, \ov X_2) \in [0,T) \times H_1 \times H_2$ and $\big(a,(p_1,p_2), \cA \big) \in \R \times (H_1 \times H_2) \times \sym(H_1 \times H_2)$ we have 
    \begin{align*}
      \big(a, (p_1,p_2), \cA \big) \in J^{2,+} U( \ov t, \ov X_1, \ov X_2). 
    \end{align*}
    Then for each $\eps, \kappa > 0$, there exist $(a_i, \cA_i) \in \R \times \sym(H_i)$ such that the following hold:
    \begin{align*}
        &\Big(a_i, p_i, \cA_i + \left(\|\cA\|_{\op} + \frac{1}{\eps}\right) Q_i \Big) \in \ov{J}^{2,+}U_i(t, \ov X_i), \quad i = 1,2,  
        \\
       &\cA_i = P_i \cA_i P_i, \quad  a_1 + a_2 = a
        \\
        &- \Big( \|\cB\|_{\op} + \frac{1}{\kappa} \Big) \text{Id}_{H_1 \times H_2} \leq \begin{pmatrix}
            \cA_1 & 0 
            \\
            0 & \cA_2
        \end{pmatrix}
        \leq \cB + \kappa \cB^2, \quad \cB = \begin{pmatrix}
            P_1 & 0 
            \\
            0 & P_2
        \end{pmatrix} \Big(\cA + \eps \cA^2 \Big) \begin{pmatrix}
            P_1 & 0 
            \\
            0 & P_2
        \end{pmatrix} .
    \end{align*}
\end{lem}

\begin{proof}
   This is a straightforward adaptation of Theorem 3.2 in \cite{CKS}, see in particular the discussion on page 5 of that paper.
\end{proof}



    
    

\subsection{Lipschitz bounds}

First, we state a uniform in $N$ Lipschitz estimate for the $N$-particle value functions $V^N$. If $V^{N}$ is smooth enough and the Hamiltonian $H$ is $C^1$, this bound can be proved in by differentiating the equation \eqref{hjb.n} as in \cite[Lemma 3.1]{CDJS}. In the present setting without idiosyncratic noise, $V^N$ is only Lipschitz, and $H$ is only locally Lispschitz, so we instead provide a proof by doubling variables in the Appendix.

\begin{prop} \label{prop.uniformlip}
    There is a constant $C_0$ such that for each $N \in \N$, the unique viscosity solution $V^N$ of \eqref{hjb.n} satisfies
    \begin{align*}
        |V^N(t,\bx) - V^N(s,\by)| \leq C_0 \Big( \bd_1(m_{\bx}^N,m_{\by}^N) + |t-s|^{1/2} \Big)
    \end{align*}
    for each $t,s \in [0,T]$, $\bx,\by \in (\R^d)^N$.
\end{prop}

The following regularity estimates are then an immediate consequence of Proposition \ref{prop.uniformlip} and the qualitative convergence of $V^N$ to $U$ as $N \to \infty$ established in \cite[Theorem 1.2]{GMS}.

\begin{cor} \label{cor.Ulip}
     There is a constant $C_0$ such that the unique $L$-viscosity solution $U$ of \eqref{hjb.inf} satisfies 
     \begin{align*}
        |U(t,m) - U(s,n)| \leq C_0 \Big(\bd_1(m,n) + |t-s|^{1/2} \Big)
     \end{align*}
     for each $t,s \in [0,T]$, $m,n \in \cP_2(\R^d)$. 
\end{cor}

Among other things, the $\mbf d_1$-Lipschitz continuity leads to useful information when the solutions are tested by smooth functions, as evinced by the following result.

\begin{lem} \label{lem.linf.touching}
    Suppose that $\Phi, \Psi : \cH \to \R$ are such that $\Phi \in C^1(\cH)$, and $\Psi$ is Lipschitz with respect to $L^1$, i.e. there is a constant $C_0$ such that
    \begin{align*}
        |\Psi(X) - \Psi(Y)| \leq C_0 \|X - Y\|_1
    \end{align*}
    for all $X,Y \in L^2$. Suppose further that, for some $X_0 \in \mcl H$,
    \begin{align*}
        \Psi(X_0) - \Phi(X_0) = \sup_{X \in L^2} \Big( \Psi(X) - \Phi(X) \Big).
    \end{align*}
    Then 
    \begin{align*}
        \| \nabla \Phi(X_0)\|_{\infty} \leq C_0.
    \end{align*}
\end{lem}

\begin{proof}
    Notice that 
    \begin{align*}
        \Phi(X_0 + Y) - \Phi(X_0) \leq \Psi(X_0 + Y) - \Psi(Y) \leq C_0 \|Y\|_1, \quad \forall \,\, Y \in \cH, 
    \end{align*}
    which easily implies that 
    \begin{align*}
        \langle \nabla \Phi(X_0), Y \rangle \leq C_0 \|Y\|_1, \quad \forall \,\, Y \in \cH, 
    \end{align*}
    and the result follows. 
\end{proof}

\subsection{Regular quantization of probability measures and random variables}
\label{sec.simquant}

Given $m \in \cP_2(\R^d)$, we denote 
\begin{align*}
    e_N(m) = \inf_{\bx \in (\R^d)^N} \bd_2\big(m,m_{\bx}^N\big).
\end{align*}
The question of estimating the decay rate of $e_N(m)$ as $N \to \infty$ is related to the quantization of probability measures, which involves optimally approximating $m$ by a measure $\hat m$ whose support contains at most $N$ points; see \cite{GL} for an extensive overview. For the general quantization problem, the work of Zador \cite{Z} gives the general rate $N^{-1/d}$ for any measure $m \in \mcl P_q$ with $q > 2$, with a constant proportional to $\mcl M_q(m)^{1/q}$. On the other hand, for the ``constrained'' quantization rate $e_N$, the rate may be worse for lower dimensions, for which the Wasserstein distance $\mbf d_2$ may be sensitive to concentration; see \cite{BX,CH}.

We will find it useful to characterize the rate $e_N$ in terms of random variables. Let $P_N$ denote the set of all partitions $\Pi = (\Omega_i)_{i = 1,\ldots,N}$ of $\Omega$ into sets of equal probability. We denote by $\cF_{\Pi}$ the (discrete) $\sigma$-algebra generated by $\Pi$. For $\Pi \in P_N$ and $\bx \in (\R^d)^N$, set \begin{align*}
    Y_{\bx, \Pi}^N = \sum_{i = 1}^N x^i 1_{\Omega^i} \in \cH.
\end{align*}
For $X \in \cH$, and $\Pi \in P_N$, denote 
\begin{align}
\label{defn:rhoN}
    \rho_N(X,\Pi) = \inf_{\bx} \| X - Y_{\bx, \Pi}^N\|_2 = \| X - \E[X | \cF_{\Pi}]\|_2 = \sqrt{\text{Var}(X | \Pi)}.
\end{align}
Notice that if $\cL(X) = m$, then
\begin{align}
\label{defn:eN}
    e_N(m) = \inf_{\Pi \in P_N} \rho_N(X, \Pi). 
\end{align}

In this paper, we only estimate $e_N(m)$ for measures $m \in \mcl P$ with bounded support, for which we use the following rates due to Chevallier \cite{CH} which are optimal: 
\begin{prop}[Chevallier \cite{CH}]\label{prop.quantization}
    There is a constant $C$ depending only on $d$ such that for any $X \in L^{\infty}(\Omega)$, we have 
    \begin{align*}
        e_N\big(\cL(X)\big) = \inf_{\Pi \in P_N} \rho_N(X, \Pi) \leq C  \|X\|_{\infty} r_{N,d}, 
    \end{align*}
    where 
    \begin{align}
        \label{def.rnd}
        r_{N,d} = \begin{cases}
            N^{-1/2} & d = 1, 
            \\
            N^{-1/2} \log(1 + N)^{1/2} & d = 2,
            \\
            N^{-1/d} & d \geq 3.
        \end{cases}
    \end{align}
\end{prop}

In the proof of Theorem \ref{thm.main}, we will also need to "simultaneously quantize" two random variables in an efficient way. This is the purpose of the following result. 

\begin{prop} \label{prop.quantization.sim}
    There exists a constant $C > 0$ depending only on $d$ such that, for any $\alpha \in (0,1)$ and $X,Y \in \linf(\Omega)$, there exists a partition $\Pi \in P_N$ with 
    \begin{align*}
        \rho_N(X, \Pi) \leq C \|X\|_{\infty} r_{\lfloor N^{\alpha} \rfloor, d} \quad \text{and} \quad
        \rho_N(Y, \Pi) \leq C \|Y\|_{\infty} r_{\lfloor N^{(1- \alpha)} \rfloor, d}.
    \end{align*}
\end{prop}

The proposition is proved by first partitioning $\Omega$ into $N_1$ events to quantize $X$, and then partitioning each event in that partition into $N_2$ events to quantize $Y$, where $N_1$ and $N_2$ are appropriately chosen integers for which $N \approx N_1 N_2$. The main idea is given in Lemma \ref{lem:multiplyN}; we first prove a lemma that aids in controlling the discrepancy $N - N_1N_2$.

\begin{lem}\label{lem:addN}
    Let $\hat N$ and $M$ be positive integers satisfying $N = \hat N + M$, and assume that
    \begin{enumerate}[(i)]
    \item $\Omega = \hat \Omega \cup \Gamma$ for two disjoint events $\hat \Omega,\Gamma \in \mcl F$ such that $\PP(\hat \Omega) = \hat N/N$ and $\PP(\Gamma) = M/N$, and
    \item The collection $\Pi$ defined by
    \[
        \Pi = \{ \hat \Omega_j , \Gamma_k : j = 1,2,\ldots, \hat N, \; k = 1,2,\ldots, M \}
    \]
    is an $N$-partition of $\Omega$ with
    \[
        \hat \Omega = \bigcup_{j=1}^{\hat N} \hat \Omega_j
        \quad \text{and} \quad
        \Gamma = \bigcup_{k = 1}^M \Gamma_k.
    \]
    \end{enumerate}
    Define $\hat \PP = \frac{N}{\hat N} \PP$ and $\QQ = \frac{N}{M} \PP$, so that 
    \[
        \hat \Pi = ( \hat \Omega_j)_{j=1}^{\hat N} \quad \text{and} \quad P = ( \Gamma_k)_{k=1}^M
    \]
    are, respectively, an $\hat N$-partition of $(\hat \Omega, \mcl F|_{\hat\Omega}, \hat \PP)$ and an $M$-partition of $(\Gamma, \mcl F|_\Gamma, \QQ)$. Then, for any $X \in \mcl H$,
    \[
        \rho_N(X,\Pi)^2 = \frac{\hat N}{N} \rho_{\hat N}(X|_{\hat \Omega},\hat \Pi)^2 + \frac{M}{N} \rho_{M}(X|_{\Gamma}, P)^2.
    \]
\end{lem}

\begin{proof}
    We compute
    \begin{align*}
        \rho_N(X, \Pi)^2
        &= \sum_{j=1}^{\hat N} \int_{\hat \Omega_j} \left|X - N \int_{\hat \Omega_j} X d\PP \right|^2 d\PP + \sum_{k=1}^{M} \int_{\Gamma_k} \left|X - N \int_{\Gamma_k} X d\PP \right|^2 d\PP  \\
        &= \frac{\hat N}{N} \sum_{j=1}^{\hat N} \int_{\hat \Omega_j} \left|X - \hat N \int_{\hat \Omega_j} X d\hat \PP \right|^2d\hat \PP + \frac{M}{N} \sum_{k=1}^{M} \int_{\Gamma_k} \left|X - M \int_{\Gamma_k} X d\QQ \right|^2d\QQ\\
        &= \frac{\hat N}{N} \rho_{\hat N}(X|_{\hat \Omega},\hat \Pi)^2 + \frac{M}{N} \rho_{M}(X|_{\Gamma}, P)^2.
    \end{align*}
\end{proof}

\begin{lem}\label{lem:multiplyN}
    There exists a constant $C>0$ depending only on $d$ such that the following holds: if $N = N_1N_2$ for positive integers $N_1$ and $N_2$, and $X ,Y \in \linf$, then there exists a partition $\Pi \in P_N$ such that
    \[
        \rho_N(X,\Pi) \le C \|X\|_{\infty} r_{N_1,d} \quad \text{and} \quad \rho_N(Y,\Pi) \le C\|Y\|_{\infty} r_{N_2,d}.
    \]
\end{lem}

\begin{proof}
First, by Proposition \ref{prop.quantization}, there exists a partition $\Pi_0 = (\Omega_i)_{i=1}^{N_1}$ of $\Omega$ into $N_1$ events of equal probability such that
    \[
        \rho_{N_1}(X,\Pi_0) \le C \norm{X}_{\infty} r_{N_1,d}.
    \]
    For fixed $i = 1,2,\ldots, N_1$, define $\PP_i = N_1 \PP|_{\Omega^i}$, so that $(\Omega_i, \PP_i)$ is a probability space. Define $Y_i = Y|_{\Omega_i}$, and choose a partition $\Pi_i = (\Omega_i^j)_{j=1}^{N_2}$ of $\Omega_i$ into $N_2$-events of $\PP_i$-probability $1/N_2$ such that
    \[
        \rho_{N_2}(Y_i, \Pi_i) \le C \norm{Y_i}_{L^{\infty}(\Omega_i, \PP_i)}r_{N_2,q_2}.
    \]
    Observe that $\PP(\Omega_i^j) = \frac{1}{N_1} \PP_i(\Omega_i^j) = \frac{1}{N}$. Then $\Pi = (\Omega_i^j)_{i=1,2,\ldots, N_1,\; j = 1,2,\ldots,N_2}$ is a partition of $\Omega$ into $N$ events of equal $\PP$-probability. Because $\Pi$ is nested within $\Pi_0$, properties of projections imply that
    \[
        \rho_N(X,\Pi) = \norm{X - \EE[X\mid \mcl F_\Pi]}_2 \le \norm{X - \EE[X \mid \mcl F_{\Pi_0}]} \le C \norm{X}_{\infty}r_{N_1,d}.
    \]
    At the same time,
    \begin{align*}
        \rho_N^2(Y,\Pi)
        &= \sum_{i=1}^{N_1} \sum_{j=1}^{N_2} \int_{\Omega_i^j} \left|Y - N\int_{\Omega_i^j} Y d\PP\right|^2 d\PP \\
        &= \frac{1}{N_1} \sum_{i=1}^{N_1} \sum_{j=1}^{N_2} \int_{\Omega_i^j} \left|Y_i - N_2 \int_{\Omega_i^j} Y_i d\PP_i \right|^2 d\PP_i\\
        &= \frac{1}{N_1} \sum_{i=1}^{N_1} \rho_{N_2}^2(Y_i, \Pi_i) \\
        &\le \frac{1}{N_1} C r_{N_2,d}^2 \sum_{i=1}^{N_1} \norm{Y_i}_{L^{\infty} (\Omega_i, \PP_i)}^2\\
        &\leq C r_{N_2,d}^2 \|Y^i\|_{\infty}^2,
    \end{align*}
    and thus we conclude.
\end{proof}

\begin{proof}[Proof of Proposition \ref{prop.quantization.sim}]
   Without loss of generality, we assume that $\alpha \geq 1/2$. We set
   \begin{align*}
      \hat{N} = \lfloor N^{\alpha} \rfloor \times \lfloor N^{1-\alpha} \rfloor, \quad M = \hat{N} - N. 
   \end{align*}
    Observe then that the integer $M$ satisfies
   \begin{align*}
       M \leq N^{\alpha} + N^{1-\alpha} + 1 \le 3 N^{\alpha}.
   \end{align*}

    Let $\hat \Omega, \Gamma \in \mcl F$ be disjoint events such that $\PP(\hat \Omega) = \hat N/N$ and $\PP(\Gamma) = M/N$. For $i= 1,2$, define the probability measures $\hat \PP = \frac{1}{\PP(\hat \Omega)} \PP$ on $\hat \Omega$ and $\QQ = \frac{1}{\PP(\Gamma)}\PP$ on $\Gamma$.
    By Lemma \ref{lem:multiplyN}, for any $\eps > 0$, there exists an $\hat N$-partition $\hat\Pi = \{\hat \Omega_{1},\hat \Omega_2,\ldots,\hat\Omega_{\hat{N}}\}$ of $(\hat \Omega,\hat\PP)$ with $\hat\bP(\hat\Omega_{i}) = \frac{1}{\hat{N}}$ (hence $\bP(\hat\Omega_{i}) = \frac{1}{N}$) such that
    \[
        \rho_{\hat N}(X|_{\hat \Omega},\hat\Pi) \le  C \norm{X}_{\infty} r_{\lfloor N^\alpha \rfloor , d}  \quad \quad \text{and} \quad
        \rho_{\hat N}(Y|_{\hat \Omega},\hat \Pi) \le \|Y\|_{\infty} r_{\lfloor N^{1-\alpha} \rfloor, d}.
    \]
    Similarly, we can find a partition $P = \{\Gamma_{1},\Gamma_2,\ldots,\Gamma_{M}\}$ of $\Gamma$ into pieces with $\bP_2(\Gamma_{i}) = \frac{1}{M}$ (hence $\bP(\Gamma_{i}) = 1/N$), with the property that
    \[
        \rho_M(X|_\Gamma,P) \le C \|X\|_{\infty} r_{M,d}.
    \]
    We also have the (trivial) estimate
    \[
        \rho_M(Y|_\Gamma,P) \le 2 \|Y\|_{\infty}. 
    \]
    Invoking Lemma \ref{lem:addN} and combining the above estimates, we find that
    \begin{align*}
        \rho_N^2(X,\Pi) &\lesssim \frac{\hat N}{N} \norm{X}_{\infty}^2 r_{\lfloor N^\alpha \rfloor, d}^2 +  \frac{M}{N}  \norm{X}_{\infty}^2 r_{M,d}^2 
       \\
       &\lesssim \norm{X}_{\infty}^2 ( r_{\lfloor N^\alpha \rfloor, d}^2 + \frac{M}{N} r_{M,d}^2)
       \\
       &\lesssim  \norm{X}_{\infty}^2 r_{\lfloor N^\alpha \rfloor, d}^2, 
    \end{align*}
    where the bound $\frac{M}{N} r_{M,d}^2 \lesssim r_{\lfloor N^\alpha \rfloor, d}^2$ is easily checked from \eqref{def.rnd} and the fact that $M \lesssim N^\alpha$. This gives the desired estimate for $\rho_N(X,\Pi)$. We then also have
    \begin{align*}
        \rho_N(Y,\Pi)^2 &\lesssim \frac{\hat{N}}{N} \norm{Y}_{\infty}^2 r_{\lfloor N^{1-\alpha} \rfloor, d}^2 + \frac{M}{N} \norm{Y}_{\infty}^2 \lesssim \|Y\|_{\infty}^2 \Big(r^2_{\lfloor N^{1-\alpha} \rfloor,d}+  N^{-(1-\alpha)} \Big)
    \end{align*}
    We conclude by using the formula \eqref{def.rnd} to observe that, in all cases, the term $N^{- (1-\alpha) }$ is dominated by the first term.
\end{proof}

\begin{rmk}
    An alternative method for simultaneously quantizing random variables $X,Y: \Omega \to \Rd$ is to observe that $(X,Y)$ is an $\RR^{2d}$-valued random variable, and hence, by Proposition \ref{prop.quantization}, there exists a partition $\Pi$ for which
    \[
        \rho_N(X,\Pi) \lesssim \norm{X}_\infty r_{N,2d} \quad \text{and} \quad \rho_N(Y,\Pi) \lesssim \norm{Y}_\infty r_{N,2d}.
    \]
    This corresponds to the case $\alpha = 1/2$ in Proposition \ref{prop.quantization.sim}, and in fact is a strict improvement in that case for $d = 1$ or $d = 2$. However, as will become clear in the proof of Theorem \ref{thm.main}, we rely on the flexibility to interpolate the fineness of the partition relative to $X$ or $Y$ through the parameter $\alpha$.
\end{rmk}

We now discuss in more detail the case $d = 1$. For $N\geq 1$, let 
\begin{align*}
    \Pi_N^{\text{reg}} = \Big\{ [0,1/N), [1/N,2/N), \ldots,  [(N-2)/N, (N-1)/N), [(N-1/N), 1] \Big\} \in P_N.
\end{align*}
We will make use of the following special property which holds when $d = 1$.

\begin{lem} \label{lem.1dquant}
    For any $m \in \cP_2(\R)$ and $\bx \in \R^N$, there exists a random variable $X$ such that $\cL\big(X, Y_{\bx, \Pi_N^{\reg}}^N\big)$ is an optimal coupling of $m$ and $m_{\bx}^N$, and
    \begin{align*}
        \norm{ X - \E\big[ X | \cF_{\Pi_N^{\text{reg}}}] }_1 \leq 2 \|X \|_{\infty} N^{-1}. 
    \end{align*}
\end{lem}

\begin{proof}
    First, suppose that $\bx = (x^1,\ldots,x^N)$ satisfies $x^1 \leq x^2 \leq \cdots \leq x^N$. Let $X : \Omega = [0,1] \to \R$ be defined by 
    \begin{align} \label{def.rearrangement}
        X(\omega) = \begin{cases}
            \inf \big\{ t \in \R : m((-\infty, t]) > \omega \big\} & \omega \in [0,1), 
            \\
           \inf \big\{ t \in \R : m((-\infty, t]) = 1 \big\} & \omega  =1.
        \end{cases}
    \end{align}
    Then $X$ is non-decreasing and $\cL(X) = m$, and in addition $X(1) = \esssup_{\omega} X$, $X(0) = \essinf_{\omega} X$. Moreover, for any $\by = (y^1,\ldots,y^N) \in \R^N$ with $y^1 \leq y^2 \leq \cdots \leq y^N$, $\cL\big(X, Y_{\by, \Pi_N^{\reg}}^N\big)$ is an optimal coupling of $m$ with $m_{\by}^N$, which is a consequence of the cyclical monotonicity of the optimal transport plan (see for instance \cite[Theorem 2.9]{Sant_15}). In particular, $\cL\big(X, Y_{\bx, \Pi_N^{\reg}}^N\big)$ is an optimal coupling of $m$ with $m_{\bx}^N$, and we have 
    \begin{align*}
        \norm{ X - \E\big[ X | \cF_{\Pi_N^{\text{reg}}}] }_1 &= \sum_{i = 1}^N \int_{(i-1)/N}^{i/N} \Big| X(\omega) - N \int_{(i-1)/N}^{i/N} X(\omega') d\omega' \Big| d \omega 
        \\
        &\leq \frac{1}{N} \sum_{i = 1}^N \Big( X\big( \frac{i}{N} \big) - X\big( \frac{i - 1}{N} \big) \Big) = \frac{X(1) - X(0)}{N} \leq \frac{2}{N} \| X\|_{\infty}. 
    \end{align*}
    Now for the general case, let $\sigma$ be a permutation of $\{1,\ldots,N\}$ such that $x^{\sigma(1)} \leq \cdots \leq x^{\sigma(N)}$. Applying the result to $\bx^{\sigma} = (x^{\sigma(1)},\ldots,x^{\sigma(N)})$, we find a random variable $X^{\sigma}$ such that $\cL(X^{\sigma}, Y_{\bx^{\sigma}}^N)$ is an optimal coupling and $e_N(m) = \rho_N(X^{\sigma}, \Pi_N^{\reg})$. Now let $T : [0,1] \to [0,1]$ be the map which carries each  $[(i-1)/N, i/N)$ onto $[(\sigma(i) - 1)/N, \sigma(i)/N)$ via an increasing linear bijection. Since $T$ is measure-preserving, the random variable $X = X^{\sigma} \circ T^{-1}$ satisfies $\mcl L(X) = m$. We also have
    \begin{align*}
        \bd_2(m, m_{\bx}^N) = \| X^{\sigma} - Y_{\bx^{\sigma}, \Pi_N^{\reg}}^N \|_2 = \| X \circ T - Y_{\bx, \Pi_N^{\reg}}^N \circ T \|_2 = \|X - Y_{\bx, \Pi_N^{\reg}}^N\|_2, 
    \end{align*}
    so that $\cL(X, Y_{\bx, \Pi_N^{\reg}}^N)$ is an optimal coupling, and similarly it is easy to check that 
    \begin{align*}
            \norm{ X - \E\big[ X | \cF_{\Pi_N^{\text{reg}}}] }_1 =     \norm{ X^{\sigma} - \E\big[ X^{\sigma} | \cF_{\Pi_N^{\text{reg}}}] }_1 \leq \frac{2}{N} \|X^{\sigma}\|_{\infty} = \frac{2}{N} \|X\|_{\infty}.
    \end{align*}
    This completes the proof. 
\end{proof}

We close this section with a lemma that explains how to compute derivatives of ``projections'' of functions $\Phi : \cH \to \R$. We omit the proof, which is a standard computation.

\begin{lem} \label{lem.projection.deriv}
    Let $\Phi \in C^1(\cH)$. Fix an $N$-partition $\Pi = \{\Omega_1,\ldots,\Omega_N\} \in P_N$, and set 
    \begin{align*}
        \Phi^N : (\R^d)^N \to \R, \quad \Phi^N(\bx) = \Phi\big( Y_{\bx, \Pi}^N \big).
    \end{align*}
    Then $\Phi^N \in C^1\big( (\R^d)^N \big)$, with 
    \begin{align*}
        D_{x^i} \Phi^N(\bx) = \E\big[ \nabla \Phi(Y_{\bx, \Pi}^N) 1_{\Omega^i} \big] = \frac{1}{N} \E\big[ \nabla \Phi^N | \Omega^i \big].
    \end{align*}
    If $\Phi \in C^2(\cH)$, then $\Phi^N \in C^2\big( (\R^d)^N \big)$, and 
    \begin{align*}
        D_{x^ix^j} \Phi^N(\bx) = \nabla^2 \Phi (Y_{\bx, \Pi}^N) \big[1_{\Omega^i}, 1_{\Omega^j} \big].
    \end{align*}
\end{lem}

\section{Error estimates}

\label{sec.main}

The goal of this section is to prove Theorem \ref{thm.main}. We split into two cases: first, when $d = 1$, for which the geometry of the real line allows for certain simplifications; and then the more involved case when $d \ge 2$.

Observe that, in view of the Lipschitz regularity estimates in Proposition \ref{prop.uniformlip} and Corollary \ref{cor.Ulip}, and in view of Lemma \ref{lem.linf.touching}, we can choose $R > 0$ such that $V^N$ is the unique viscosity solution of
\begin{align} \label{hjb.n.R} \tag{$\text{HJB}_{N,R}$}
    \begin{cases}
        \ds - \partial_t V^N - \kappa \sum_{i,j = 1}^N \tr\big(D_{x^ix^j} V^N \big) + \frac{1}{N} \sum_{i = 1}^N H^R\big(x^i, ND_{x^i} V^N, m_{\bx}^N\big) = 0 \quad (t,\bx) \in [0,T) \times (\R^d)^N, \vspace{.2cm}
        \\
       \ds  V^N(T,\bx) = G(m_{\bx}^N).
    \end{cases}
\end{align}
and $U$ is the unique viscosity solution of 
\begin{align} \label{hjb.inf.R} \tag{$\text{HJB}_{\infty,R}$}
   \begin{cases} 
   \ds - \partial_t U(t,m) - \kappa \com U(t,m) + \int_{\R^d} H^R\big(x, D_mU(t,m,x),m \big) m(dx) = 0, \quad (t,m) \in [0,T) \times \cP_2(\R^d), \vspace{.2cm}
   \\
   \ds  U(T,m) = G(m), \quad m \in \cP_2(\R^d),
   \end{cases}
\end{align}
where 
\begin{align*}
    H^R(x,p,m) = H(x,\pi_R(m),m), 
\end{align*}
and $\pi_R : \R^d \to \R^d$ is the orthogonal projection onto the ball of radius $R$. Since $H^R$ is (globally) Lipschitz continuous thanks to Assumption \ref{assump.main}, we can assume (up to a relabeling of the Hamiltonian) throughout this Section that $H$ satisfies 
\[
    |H(x, p, m) - H(x',p', m')| \le C_H( |x- x'| + |p - p'| + \mbf d_1(m, m') ), 
    \quad x,x',p,p' \in \Rd, \; m,m' \in \mcl P_2.
\]

Throughout the section, for $a,b \in \R$, we use the notation $a \lesssim b$ to mean thatn $a \le Cb$ for some constant $C > 0$ that depends only on universal quantities, the dimension $d$, and the constants $C_H$ and $C_G$ in Assumption \ref{assump.main}, and is independent of $N$ as well as any of the small parameters (which themselves will be chosen optimally to depend on $N$) introduced in the doubling-of-variables arguments below.

\subsection{The case $d = 1$}

\label{subsec.d1}

For $N \in \N$, and $\eps, \eta, \lambda, \gamma > 0$, we consider the function $\Psi^N_{\eps, \eta, \lambda, \gamma} : [0,T] \times \cP_2 \times \R^N \to \R$, defined by
\begin{align} 
    \Psi_{\eps, \eta, \lambda, \gamma}^N(t,m,\bx) &= U(t,m) - V^N(t,\bx) - \frac{1}{2\eps} \bd_2^2\big(m,m_{\bx}^N \big) \notag \\
    &- \frac{\eta}{2} \big(1 + T-t\big) M_2(m_{\bx}^N) - \lambda(T-t) - \gamma \Big( \frac{1}{t} - \frac{1}{T} \Big).\label{def.psin}
\end{align}

\begin{lem}  \label{lem.ovX}
    For each $\eps, \eta, \lambda, \gamma > 0$, there exists a maximum point $(\ov{t}, \ov{m}, \ov{\bx})$ of $\Psi_{\eps, \eta, \lambda, \gamma}^N$.
\end{lem}

\begin{proof} The proof is a simpler version of the proof of Lemma \ref{lem.fistpenalization} below, and so is omitted. 
\end{proof}

\begin{prop} \label{prop.1d.lambdabound}
There is a constant $C_0$ with the following property: if $\eps, \lambda, \eta, \gamma > 0$ and there exists a maximum point $(\ov{t}, \ov{m}, \ov{\bx})$ of $\Psi_{\eps, \eta, \lambda, \gamma}^N$ with $\ov{t} < T$, then 
\begin{align*}
    \lambda \leq  C_0\Big(\eps + \eta + \frac{1}{\eps \eta N} \Big).
\end{align*}
\end{prop}

\begin{proof}
Let $(\ov{t}, \ov{m}, \ov{\bx})$ be as in the statement of Lemma \ref{lem.ovX}. Applying Lemma \ref{lem.1dquant}, we find $\ov{X}$ with $\cL(\ov{X}) = \ov{m}$ such that $\cL\big( \ov{X}, Y_{\ov{\bx}, \Pi_N^{\reg}}^N)$ is an optimal coupling of $\ov{m}$ and $m_{\ov{\bx}}^N$, and 
\begin{align} \label{ovx.l1bound}
    \norm{ \ov X - \E\big[ \ov X | \cF_{\Pi_N}^{\reg} \big] }_1 \leq \frac{2}{N} \|\ov{X}\|_{\infty}.
\end{align}
It follows that $(\ov{t}, \ov{\bX}, \ov{\bx})$ is a maximum point of the function
    \begin{align*}
        [0,T] \times \cH \times \R^N \ni (t,X,\bx) &\mapsto \hat{U}(t,X) - V^N(t,\bx)  - \frac{1}{2\eps} \|X-Y_{\bx}^N\|_2^2 \\
        &- \frac{\eta}{2} (1 + T - t) \|Y_{\bx}^N\|_2^2 - \lambda(T-t) - \gamma\Big(\frac{1}{t} - \frac{1}{T}\Big), 
    \end{align*}
    where here and for the remainder of the argument we set $Y_{\bx}^N = Y_{\bx, \Pi_N^{\reg}}^N$ for simplicity, and $\Pi_N^{\reg}$ is defined in Subsection \ref{sec.simquant}. We next claim that 
    \begin{align}
    \label{eq:claim01/10}
        \norm{ \frac{\ov X - Y_{\ov \bx}^N}{\eps} }_{\infty} \lesssim 1, \quad \norm{\ov X}_{\infty} + \norm{\ov Y_{\ov{\bx}^N}}_{\infty} \lesssim 1 + \eta^{-1}.
    \end{align}
    Indeed, the first bound comes from Lemma \ref{lem.linf.touching} and the $L^1$-Lipschitz regularity of $\hat{U}$, and then the second comes from using the spatial regularity from Proposition \ref{prop.uniformlip} together with Lemma  \ref{lem.projection.deriv} (to compute the gradient of $\bx \mapsto \norm{X - Y_{\bx}^N}_2^2 $ and $\bx \mapsto \|Y_{\bx}^N\|_2^2$) to find
    \begin{align*}
        \Big\| &\E\Big[ \frac{ \ov X - Y_{\ov \bx}^N}{\eps}  | \cF_{\Pi_N^{\reg}}\Big] + \eta(1 + T - \ov t) Y_{\ov{\bx}}^N \Big\|_{\infty} \lesssim 1 
        \\
        &\implies \eta \| Y_{\ov \bx}^N\|_{\infty} 
        \lesssim 1 + \norm{ \frac{ \ov X - Y_{\ov \bx}^N}{\eps} }_{\infty} \lesssim 1
        \\
        &\implies \|\ov X \|_{\infty} \lesssim \eps + \|Y_{\ov \bx}\|_{\infty} \lesssim 1 + \eta^{-1}.
    \end{align*}   
    Now we define $\mcl A_N \in \text{Sym}(\mcl H \times \R^N)$ by
    \begin{align*}
        \cA_N[(X,\bx), (Y,\by)] = \frac{1}{\eps} \langle X - Y_{\bx}^N, Y - Y_{\by}^N \rangle_{\cH}, \quad (X,\bx), (Y,\by) \in \cH \times \R^N.
    \end{align*}
    A direct computation shows that we can also express $\cA_N$ as
\begin{equation}
    \label{eq:defnA_N}
        \mcl A_N[ (X,\bx), (Y,\by)]
        := 
        \frac{1}{\eps}
	\left\langle
	\begin{pmatrix}
		\Id_{\mcl H} & -\mbf Q^N \\
		-(\mbf Q^N)^* & N^{-1} \Id_N
	\end{pmatrix}
	\begin{pmatrix}
	X \\
	\bx
	\end{pmatrix}
	,
	\begin{pmatrix}
		Y\\
		\by
	\end{pmatrix}
	\right\rangle, 
    \quad (X,\bx), (Y,\by) \in \mcl H \times \R^N,
\end{equation}
    where $\mbf Q^N : \R^N \to \cH$ and $(\bm Q^N)^* : \cH \to \R^N$ are bounded linear operators defined by 
    \begin{align*}
        \bm Q^N \bx = Y_{\bx}^N, \quad (\bm Q^N)^* X = \Big( \E\big[X 1_{[0,1/N)} \big], \E\big[X 1_{[1/N,2/N)} \big],\ldots, \E\big[X 1_{[(N-1)/N, 1)}\big] \Big). 
    \end{align*}
    We now set $\cH_0 = \text{span}(1)$, i.e. $\cH_0$ is the one-dimensional subspace of $\cH$ consisting of all constant random variables. We define $P_{\cH_0}  : \cH \to \cH$ to be the orthogonal projection onto $\cH_0$, and set $P_{\cH_0^{\perp}} = \text{Id}_{\cH} - P_{\cH_0}$. 
    We now suppose that $\ov{t} < T$, and apply Lemma \ref{lem.cks} to the maximum $(\oline t, \oline X, \oline \bx)$, taking $H_1 = \cH$, $Z_1 = \cH_0$, $H_2 = \R^N, Z_2 = \R^N$, $U_1 = U$, $U_2 = - V^N$, and
    \begin{align*}
        &a = - \lambda - \frac{\eta}{2} \|Y_{\ov{\bx}}^N\|_2^2 - \gamma \ov{t}^{-2}
        \\
        & (p_1,p_2) = \Big( \frac{\ov X - Y_{\ov \bx}^N}{\eps}, \Big( -\frac{1}{N} \EE\left[ \frac{\oline X - Y^N_{\oline\bx}}{\eps} \mid \Omega_i \right]  + \frac{\eta}{N} (1 + T - \ov t) \ov{x}^i \Big)_{i=1}^N\Big)
        \\
        &\cA = \cA_N, 
    \end{align*}
    and $\kappa = \eps$. Observe that the assumption \eqref{BT} introduced before Lemma \ref{lem.cks} is satisfied here in view of the fact that $\hat U$ and $V^N$ solve evolution equations (see \cite[Remark 3.1]{CKS}).

    Using the fact that $\ov{J}^{2,-}V^N = - \ov{J}^{2,+}(-V^N)$ and the fact that $\norm{\cA}_{\text{op}} \leq \frac{2}{\eps}$, we deduce that there exist $a_1,a_2 \in \R$, $A_1 \in  \text{Sym}(\cH)$ and $A_2 \in  \text{Sym}(\RR^{N})$ such that $P_{\cH_0} A_1 P_{\cH_0} = A_1$ (in particular, since $\cH_0$ is 1-dimensional, $A_1$ can be represented by the single number $A_1 \cdot 1$), 
\[
	\left(a_1, \,  \frac{\oline X - Y^N_{\oline \bx}}{\eps}, \, A_1 + \frac{3}{\eps} P_{\cH_0^{\perp}} \right) \in \oline{J}^{2,+}\hat U(\oline t,\oline X),
\]
\[
	\left(a_2, \,  \left( \frac{1}{N} \EE\left[ \frac{\oline X - Y^N_{\oline\bx}}{\eps} \mid \Omega_i \right]  - \frac{\eta}{N} (1 + T - \ov t) \ov{x}^i \right)_{i=1}^N, \,  A_2 \right) \in \oline{J}^{2,-} V^N(\oline t,\oline \bx),
\]
\[
    a_1 - a_2 = - \lambda - \frac{\eta}{2} \norm{Y_{\ov \bx}^N}_2^2 - \gamma \ov{t}^{-2}  
\]
and
\begin{equation}\label{mixedmatrix}
	\begin{pmatrix}
		A_1 & 0 \\
		0 & -A_2
	\end{pmatrix}
	\le \cB + \eps \cB^2, 
\end{equation}
where 
\begin{equation}
\cB =	\begin{pmatrix}
		P_{\cH_0} & 0 \\
		0 & \Id_{\R^N}
	\end{pmatrix}
	(\mcl A_N + \eps\mcl A_N^2)
	\begin{pmatrix}
		P_{\cH_0} & 0 \\
		0 & \Id_{\R^N}
	\end{pmatrix} .
\end{equation}
Note that we have used Lemma \ref{lem.projection.deriv} to characterize the subjets of $V^N$. 

Applying the $L$-viscosity subsolution definition for $\hat U$ and using the fact that $P_{\mathcal{H}_0^{\perp}}[1,1] = 0$, we find
\[
	- a  - \kappa  A_1  [1,1]  + \hat H\left( \oline X, \frac{\oline X - Y^N_{\oline \bx} }{\eps}\right)  \le 0,
\]
while the supersolution definition for $V^N$ together with Lemma \ref{lem.projection.deriv} gives
\begin{align*}
	 &-a_2 - \kappa \sum_{1 \le i,j \le N}  (A_2)_{i,j} +  \hat H\left(Y^N_{\oline x},  \EE\left[  \frac{\oline X - Y^N_{\oline \bx}}{\eps}\,  \big| \, \mcl F_\Pi \right]    - \eta (1 + T - \ov{t}) Y_{\ov \bx}^N \right) 
     \\
     &=-a_2 - \kappa \sum_{1 \le i,j \le N}   (A_2)_{i,j} +  \frac{1}{N} \sum_{i = 1}^N H\Big(\ov{x}^i, \E\Big[ \frac{\ov{X} - Y_{\ov{\bx}}^N}{\eps} | \Omega^i\Big] - \eta\Big(1 + T - \ov{t}\Big) \ov{x}^i, m_{\ov{\bx}}^N \Big)
    	\ge 0.
\end{align*}
Set $\bm{1} = \big(1, (1,\ldots,1) \big) \in \cH \times \R^N$, and notice that $\cA \bm{1} = 0$, hence $\cB \bm{1} = 0$. As a consequence, the matrix inequality \eqref{mixedmatrix} implies
\begin{align*}
	 A_1[1,1] - \sum_{i,j=1}^N   (A_2)_{ij} = \Big \langle \bm{1}, \,\,  \begin{pmatrix} A_1 & 0 
    \\
     0 & -A_2
	\end{pmatrix} \bm{1} \Big \rangle_{\cH \times \R^N} \leq \big \langle \bm{1}, \, \big(\cB + \eps \cB^2\big) \bm{1} \rangle_{\cH \times \R^N} = 0. 
\end{align*}
Therefore, subtracting the inequalities for $\hat U$ and $V^N$, we get
\begin{align*}
   \lambda + \frac{\eta}{2} \norm{Y^N_{\ov\bx}}_2^2 &\leq -a_1 + a_2
   \\
    &\leq - \hat H\left( \oline X, \frac{\oline X - Y^N_{\oline \bx} }{\eps}  \right)
    + \hat H\left(Y^N_{\oline x},  \EE\left[  \frac{\oline X - Y^N_{\oline \bx}}{\eps}\,  \big| \, \mcl F_\Pi \right] - \eta (1 + T - \ov{t}) Y_{\ov \bx}^N \right).
\end{align*}
The Lipschitz regularity of $\hat H$ together with \eqref{eq:claim01/10} implies that
\begin{align*}
    \hat H\left(Y^N_{\oline x},  \EE\left[  \frac{\oline X - Y^N_{\oline \bx}}{\eps}\,  \big| \, \mcl F_\Pi \right] - \eta (1 + T - \ov{t}) Y_{\ov \bx}^N \right)
     &- \hat H\left(\ov X,  \EE\left[  \frac{\oline X - Y^N_{\oline \bx}}{\eps}\,  \big| \, \mcl F_\Pi \right] \right) \\
     &\lesssim  \norm{\oline X - Y^N_{\oline \bx}}_2 + \eta \|Y_{\ov{\bx}}^N \|_2  \\
     &\lesssim \eps + \eta \|Y_{\ov{\bx}}^N\|_2
\end{align*}
and the Lipschitz regularity of $H$ in the $p$-variable together with \eqref{ovx.l1bound} gives
\begin{align*}
   \hat{H}\Big( &\ov X, \E\Big[ \frac{\oline X - Y^N_{\oline \bx} }{\eps} \mid \cF_{\Pi} \Big]\Big) - \hat{H}\Big( \ov X, \frac{\oline X - Y^N_{\oline \bx}}{\eps} \Big) 
   \lesssim \frac{1}{\eps} \| \ov{X} - \E[\ov{X} | \cF_{\Pi}] \|_1 \lesssim \frac{1}{\eps N} \| \ov X\|_{\infty} \lesssim \frac{1}{\eps \eta N}.
\end{align*}
Putting these estimates together, we deduce that
\[
    \lambda + \frac{\eta}{2} \norm{Y_{\ov \bx}^N}_2^2
    \lesssim \eps + \frac{1}{\eps \eta N} + \eta\|Y_{\ov \bx}\|_2.
\]
An application of Young's inequality completes the proof.
\end{proof}

We are now ready to complete the proof of Theorem \ref{thm.main} in the case $d = 1$. 

\begin{proof}[Proof of Theorem \ref{thm.main}, $d = 1$]
    By Proposition \ref{prop.1d.lambdabound}, there is a constant $C_0$ such that if 
    \begin{align*}
        \lambda = C_0\Big(\eps + \eta + \frac{1}{N \eps \eta}\Big), 
    \end{align*}
    then $\ov{t} = T$. Thus with this choice of $\lambda$, we have 
    \begin{align*}
        U(t,m_{\bx}^N) - V^N(t,\bx) &\lesssim 
        U(T, \ov{m}) - V^N(T, \ov{\bx}) - \frac{1}{2\eps} \bd_2^2(\ov{m}, m_{\ov{\bx}}^N) + \eta\big(1 + M_2(m_{\bx}^N) \big) + \lambda + \gamma\Big( \frac{1}{t} - \frac{1}{T} \Big)
        \\
        &\lesssim \eps + \eta\big(1 + M_2(m_{\bx}^N) \big) + \eta^{-1} \eps^{-1} N^{-1} + \gamma t^{-1}.
    \end{align*}
    Sending $\gamma \downarrow 0$ and then choosing $\eps = \eta = N^{-1/3}$, we obtain the result. 
\end{proof}

\subsection{The case $d \ge 2$}

\label{subsec.main.ge2}

We focus on the upper bound for $U(t,m_{\bx}^N) - V^N(t,\bx)$, the other direction having a symmetric proof. Recall that by Proposition \ref{prop.uniformlip}, there is a constant $C_0$ independent of $N$ such that 
\begin{align*}
    |V^N(t,\bx) - V^N(s,\by)| \leq C_0 \Big(\bd_1\big(m_{\bx}^N, m_{\by}^N\big) + |t-s|^{1/2} \Big). 
\end{align*}
As a consequence, the map $\wt{V}^N : [0,T] \times \cP_2 \to \R$, defined by 
\begin{align} \label{def.tildeVn}
    \wt{V}^N(t,m) = \inf_{\bx \in (\R^d)^N} \Big\{ V^N(t,\bx) + 2C_0 \bd_1\big(m_{\bx}^N, m \big) \Big\}
\end{align}
satisfies 
\begin{align*}
    \tilde{V}^N(t,m_{\bx}^N) = V^N(t,\bx), \quad |\tilde{V}^N(t,m) - \tilde{V}^N(s,n)| \leq 2C_0 \Big( \bd_1(m,n) + |t-s|^{1/2} \Big), 
\end{align*}
for each $t,s \in [0,T]$, $\bx \in (\R^d)^N$, $m,n \in \cP_2$. It is also straightforward to check that \eqref{def.tildeVn} admits at least one minimizer. We define the lifts of $\wt{V}^N$ and $U$ to $\cH$ by
\begin{equation}
\label{eq:defnVhatNUhat17/09}
    \hat{V}^N, \hat{U} : [0,T] \times \cH \to \R, \quad \hat{V}^N(t,X) = \wt{V}^N(t,\cL(X)), \quad \hat{U}(t,X) = U(t,\cL(X)).
\end{equation}

\subsubsection{Doubling of variables and estimates from penalization}
Given parameters $\eps, \eta, \lambda, \gamma > 0$, and $N \in \N$, we define a map $\Phi_{\eps, \eta, \lambda, \gamma}^N : (0,T] \times \cP_2 \times \cP_2 \to \R$ via the formula (recall the notation $M_2(m)$ from \eqref{defn:M2})
\begin{align} \label{def.phin}
    \Phi_{\eps, \eta, \lambda, \gamma}^N(t,m,n) &= U(t,m) - \wt{V}(t,n) - \frac{1}{2\eps} \bd_2^2\big(m,n\big) \notag\\
    &- \frac{\eta}{2} \big(1 + T -t \big) M_2(m) - \lambda (T-t) - \gamma \Big( \frac{1}{t} - \frac{1}{T} \Big).
\end{align}
We also define $\hat{\Phi}_{\eps, \eta, \lambda, \gamma}^N : [0,T] \times \cH \times \cH \to \R$ by
\begin{align}\label{def.phin.lift}
    \hat{\Phi}_{\eps, \eta, \lambda , \gamma}^N(t,X,Y) &= \hat{U}(t,X) - \hat{V}^N(t,Y) - \frac{1}{2\eps} \|X - Y\|_2^2 \notag\\
    &- \frac{\eta}{2} \big(1 + T-t \big) \|X\|_2^2 - \lambda (T-t) - \gamma \Big( \frac{1}{t} - \frac{1}{T} \Big).
\end{align}
At times we write simply $\Phi^N = \Phi_{\eps, \eta, \lambda, \gamma}^N$ or $\hat{\Phi}^N = \hat{\Phi}_{\eps, \eta, \lambda, \gamma}^N$. 

We first establish the existence of maximum points of $\Phi^N$ and $\hat{\Phi}^N$ and obtain useful estimates.

\begin{lem} \label{lem.fistpenalization}
    For each $\eps, \eta, \lambda, \gamma > 0$ and $N \in \N$, the function $\Phi^N$ admits at least one maximum point $(\ov{t}, \ov{m}, \ov{n})$. Moreover, if $ \ov{X}, \ov{Y} \in \cH$ are such that $\mcl L(\oline{X}, \oline{Y})$ is an optimal coupling of $\ov{m}$ and $\ov{n}$, i.e. 
    \begin{align*}
        \cL(\ov{X}) = \ov{m}, \quad \cL(\ov{Y}) = \ov{n}, \quad \bd_2\big(\ov{m}, \ov{n}) = \|\ov{X} - \ov{Y}\|_2, 
    \end{align*}
    then $(\ov{t}, \ov{X}, \ov{Y})$ is a maximum point of $\hat{\Phi}_{\eps, \eta, \lambda,\gamma}^N$. Moreover, there is a constant $C$ independent of $\eps, \eta, \lambda, \gamma$, and $N$ such that any maximum point of $\hat{\Phi}_{\eps, \eta, \lambda,\gamma}^N$ satisfies 
    \begin{align} \label{penalization.bounds}
        \norm{\frac{\ov{X} - \ov{Y}}{\eps}}_{\infty} + \eta \norm{ \ov{X}}_{\infty} \leq C, 
    \end{align}
    and hence $\norm{\ov Y}_{\infty} \leq C(\eta^{-1} + \eps)$.
\end{lem}

\begin{proof}
    Let $(t_j, m_j, n_j)_{j \in \N}$ be a maximizing sequence for $\Phi_{\eps, \eta,\lambda, \gamma}^N$. Clearly, $m_j$ and $n_j$ are both bounded in $\bd_2$. It follows that there exists $\ov{t}, \ov{m}, \ov{n}$ such that (along a subsequence, which we relabel as the original sequence), 
    \begin{align*}
        t_j \to \ov{t}, \quad \bd_1(m_j, \ov{m}) \to 0, \quad \bd_1(n_j, \ov{n}) \to 0. 
    \end{align*}
    We now claim that $(\ov{t}, \ov{m}, \ov{n})$ is a maximum point of $\Phi^N$. Indeed, let $\gamma_j \in \cP_2(\R^d \times \R^d)$ be an optimal coupling of $m_j$ and $n_j$. Then $\gamma_j$ is bounded in $\cP_2(\R^d \times \R^d)$, and so converges in $\bd_1$ (along a subsequence which we again relabel) to $\ov{\gamma}$, which is necessarily a coupling of $\ov{m}$ and $\ov{n}$, and so 
    \begin{align*}
        \bd_2(\ov{m},\ov{n}) \leq \int_{\R^d \times \R^d} |x-y|^2 \ov{\gamma}(dx,dy) \leq \liminf_{j \to \infty} \int_{\R^d \times \R^d} |x-y|^2 \gamma_j(dx,dy) = \liminf_{j \to \infty} \bd_2^2(m_j, n_j). 
    \end{align*}
    Similarly, we have 
    \begin{align*}
        M_2(\ov{m}) \leq \liminf_{j \to \infty} M_2(m_j).
    \end{align*}
    Together with the $\bd_1$-continuity of $U$ and $\wt{V}$, this is enough to conclude that $(\ov{t}, \ov{m}, \ov{n})$ is a maximum point, as desired. 

    Now let $(\ov{X}, \ov{Y})$ be an optimal coupling of $\ov{m}$ and $\ov{n}$. The fact that $(\ov{t}, \ov{X}, \ov{Y})$ is a maximum point of $\hat{\Phi}^N$ is an easy consequence of the definition of $\bd_2$. Finally, for the bound \eqref{penalization.bounds}, we use Lemma \ref{lem.linf.touching} and the Lipschitz continuity of both $\hat{U}$ and $\hat{V}^N$ to deduce that
    \begin{align*}
       \norm{ \frac{\ov{X} - \ov{Y}}{\eps} }_{\infty} \leq \text{Lip}(\hat{V}; L^1) \leq \text{Lip}(\wt{V}; \bd_1), 
    \end{align*}
    and likewise 
    \begin{align*}
        \norm{ \frac{\ov{X} - \ov{Y}}{\eps} + \eta(1 + T - \ov{t}) \ov{X}}_{\infty} \leq \text{Lip}(\hat{U}; L^1) \leq \text{Lip}(U; \bd_1),
    \end{align*}
    which together give the desired bound. 
\end{proof}

Throughout the rest of the section, we use $(\ov{t}, \ov{X}, \ov{Y})$ to denote some maximum pint of $\hat{\Phi}^N$, and $C$ indicates a constant which can change from line to line but which does not depend on $\eps, \eta, \lambda, \gamma$, or $N$. 
Our strategy is to bound $\Phi^N_{\eps, \eta, \lambda, \gamma}$ (or equivalently $\hat{\Phi}_{\eps, \eta, \lambda, \gamma}^N$) from above in terms of the parameters $\eps, \eta, \lambda, \gamma$. This will then give rise to the desired error estimates, because by design, we have 
\begin{align} \label{startingest}
    U(t,m_{\bx}^N) - V^N(t,\bx) \lesssim \max_{t,X,Y} \hat{\Phi}_{\eps, \eta, \lambda,\gamma}^N + \eta M_2(m_{\bx}^N) + \lambda + \gamma t^{-1}.
\end{align}

\subsubsection{Small time estimate} 

The main strategy for establishing the upper bound on $\Phi^N$ is to deduce that the extreme point $\ov t$ must lie sufficiently close to $T$, in a quantitative manner depending on $N$. This is done by appropriately choosing the various parameters and using the relationship between the equations satisfied by $U$ and $V^N$, the details for doing so being laid out in the forthcoming subsections.

Once the upper bound on $T - \ov t$ is concluded, we then may invoke the following cruder estimates that simply use the fact that $U(T,m^N_\bx) = V^N(T,\bx)$ for all $N$ and $\bx \in (\Rd)^N$.

\begin{lem}\label{lem:doubled_small_time}
    There exists a constant $C > 0$ depending only on $C_H$ and $C_G$ such that, for any $\eps > 0$, $N \in \NN$,  and $t \in [0,T]$,
    \[
        \sup_{(\mu,\bx) \in \mcl P_2 \times (\R^d)^N} \left\{ U(t, \mu) - V^N(t, \bx) - \frac{1}{2\eps} \mbf d_2^2(\mu, \mu^N_\bx) \right\} 
        \le C(T-t) + \frac{C_G^2\eps}{2}.
    \]
\end{lem}

\begin{proof}
    We introduce the viscosity solution $V^{N,\circ}:[0,T] \times (\R^d)^N \rightarrow \R$ to 
\[
    - \partial_t V^{N, \circ} - \kappa \sum_{i,j=1}^N \tr(D_{x^{i}x^{j}}V^{N,\circ})  = 0, \quad (t,\bx) \in [0,T) \times (\R^d)^N, \quad V^{N,\circ}(T,\bx) = G(m^N_{\bx}), 
\]
which has the explicit solution 
\begin{align} \label{vno.explicit}
    V^{N,\circ}(t,\bx) = \E\Big[ G\big( \frac{1}{N} \sum_{i = 1}^N \delta_{x^i + B_{T-t}} \big) \Big] =  \E\Big[ G\Big( (\text{Id} + B_{T-t})_{\#} m_{\bx}^N\Big) \Big], 
\end{align}
with $B$ a $d$-dimensional Brownian motion.
We also introduce the $L$-viscosity solution $U^{\circ} :[0,T] \times \mathcal{P}_2(\R^d) \rightarrow \R$ to
\[
    -\partial_t U^{\circ} - \kappa \Delta_{w}U^{\circ} = 0 \quad (t,m) \in [0,T] \times \mathcal{P}_2(\R^d), \quad U^{{\circ}}(T,m) = G(m),
\]
whose explicit solution is 
\begin{align} \label{uo.explicit}
    U(t,m) = \E\Big[ G\Big( (\text{Id} + B_{T-t})_{\#} m \Big)  \Big].
\end{align}
By Proposition \ref{prop.uniformlip}, and the comparison principle, if $C$ is sufficiently large depending only on the constants in Assumption \eqref{assump.main}, then
\[
    V^N \geq V^{N,\circ} - C(T-t) \quad \text{and} \quad U \leq U^\circ + C(T-t).
\]
Note that from the explicit formulas \eqref{vno.explicit} and \eqref{uo.explicit}, we have 
\begin{align*}
    V^{N,o}(t,\bx) = U^o(t,m_{\bx}^N).
\end{align*}
It follows that 
\begin{align*}
     \sup_{(\mu, \bx) \in \mcl P_2 \times (\R^d)^N}
    &\left\{ U^\circ(t,\mu) - V^{N,\circ}(t,\bx) - \frac{1}{2\eps} \mbf d_2^2(\mu, \mu^N_\bx) \right\}
    \\
    &\leq \sup_{\mu,\nu \in \mcl P_2} \left\{ U^\circ(t,\mu) - U^\circ(t,\nu) - \frac{1}{2\eps} \mbf d_2^2(\mu,\nu) \right\}
    \\
    &= \sup_{\mu,\nu} \Big\{ \E\Big[ G\Big( (\text{Id} + B_{T-t})_{\#} \mu \Big)  - G\Big((\text{Id} + B_{T-t})_{\#} \nu \Big) \Big] - \frac{1}{2\eps} \bd_2^2(\mu,\nu) \Big\}
    \\
    &= \sup_{\mu,\nu} \Big\{ \E\Big[ G\Big( (\text{Id} + B_{T-t})_{\#} \mu \Big)  - G\Big((\text{Id} + B_{T-t})_{\#} \nu \Big) 
    \\
    &\qquad \qquad - \frac{1}{2\eps} \bd_2^2\Big( \big(\text{Id} + B_{T-t}\big)_{\#}\mu, \big(\text{Id} + B_{T-t}\big)_{\#} \nu \Big) \Big] \Big\}
    \\
    &\leq \sup_{\mu,\nu} \Big\{ G(\mu) - G(\nu) - \frac{1}{2\eps} \bd_2^2(\mu,\nu) \Big\} \leq \frac{C_G^2 \eps}{2}, 
\end{align*}
where we used the fact that the for $z \in \R^d$, $\mu,\nu \in \cP_2(\R^d)$, we have $\bd_2\big( (\text{Id} + z)_{\#} \mu, (\text{Id} + z)_{\#} \nu\big) = \bd_2(\mu,\nu)$, and in the last line we used the Lipschitz continuity of $G$. 
\end{proof}

We then have the following corollary regarding the upper bound for $\Phi^N$ (equivalently $\hat \Phi^N$) in terms of the size of $T - \ov t$.

\begin{lem} \label{lem.shorttime.higherdim}
    With $r_{N,d}$ as in \eqref{def.rnd}, we have
   \begin{align*}
      \max_{m,n} \Phi_{\eps, \eta, \lambda, \gamma}^N(t,m,n) = \max_{X,Y} \hat{\Phi}_{\eps, \eta, \lambda, \gamma}^N(t,X,Y) \lesssim \big(T - \ov{t} \big) + \eps + (1 + \eta^{-1} + \eps^{-1} \eta^{-1}) r_{N,d}.
   \end{align*}
\end{lem}

\begin{proof}
    We have 
    \begin{align*}
        \max_{t,m,n} \hat{\Phi}_{\eps, \eta, \lambda, \gamma}^N = \hat{U}(\ov{t}, \ov{X}) - \hat{V}^N(\ov{t}, \ov{Y}) - \frac{1}{2\eps} \| \ov{X} - \ov{Y} \|_2^2 - \frac{\eta}{2} (1 + T - \ov{t}) \|\ov{X}\|_2^2 - \lambda(T - \ov{t}) - \gamma \Big( \frac{1}{ \ov t} - \frac{1}{T} \Big).
    \end{align*}
    By Proposition \ref{prop.quantization} and Lemma \ref{lem.fistpenalization}, we can find $\Pi \in P_N$ such that
    \begin{align*}
        &\|\ov{Y} - \E[\ov{Y} | \cF_{\Pi}]\|_1 \leq   \|\ov{Y} - \E[\ov{Y} | \cF_{\Pi}]\|_2 \lesssim \|\ov{Y}\|_{\infty} r_{N,d} \lesssim (1 + \eta^{-1}) r_{N,d}.
    \end{align*}
    We can then use the Lipschitz regularity of $\hat{V}$, and the fact that 
    \begin{align*}
       \norm{\ov{X}}_{\infty},  \norm{\ov{Y}}_{\infty}, \norm{\E[\ov{Y} | \cF_{\Pi}]}_{\infty} \lesssim 1 + \eta^{-1}, 
    \end{align*}
    to get
    \begin{align*}
        \hat{\Phi}_{\eps, \eta, \lambda,\gamma}^N(\ov t, \ov X, \ov Y) &\leq \hat{U}(\ov{t}, \ov X ) - \hat{V}^N(\ov{t}, \E[\ov{Y} | \cF_{\Pi}]) - \frac{1}{2\eps}   \| \ov X - \E[ \ov{Y} | \cF_{\Pi} ] \|_2^2
        \\
        &\qquad + \hat{V}^N(\ov t, \E[ \ov Y | \cF_{\Pi}]) - \hat{V}^N(\ov{t}, \ov Y) + \frac{1}{2\eps}   \| \ov X - \E[ \ov{Y} | \cF_{\Pi} ] \|_2^2 - \frac{1}{2\eps}   \| \ov X - \ov Y \|_2^2
        \\
        &\leq \hat{U}(\ov t, \ov X) - V^N(\ov t, \ov \by) - \frac{1}{2\eps} \bd_2^2( \ov{m}, \ov{\by}) + C (1 + \eta^{-1} + \eta^{-1} \eps^{-1}) r_{N,d}
    \end{align*}
    where 
    \begin{align*}
        \ov{m} = \cL( \ov{\bX} ), \quad \ov{x}^i = \E[ \ov{Y} | \Pi^i]. 
    \end{align*}
    Applying Lemma \ref{lem:doubled_small_time} completes the proof. 
\end{proof}

\subsubsection{Pushing extrema onto discrete random variables} Having extended $V^N$ to the full space of measures, we next introduce a penalization that pushes the extreme point $\ov Y$ to a nearby discrete random variable, allowing the viscosity solution definition to be applied to $V^N$.

We start with a technical result which will be needed in the application of the particular parabolic Crandall-Ishii lemma used in the proof of Proposition \ref{prop.1d.lambdabound} below. 

\begin{lem} \label{lem.hatvn.cks}
    The function $ - \hat{V}^N$, defined in \eqref{eq:defnVhatNUhat17/09}, satisfies the condition \eqref{BT} from Lemma \ref{lem.cks}. 
\end{lem}

\begin{proof}
    First, notice that we can write 
    \begin{align} \label{hatVn.rep}
        \hat{V}^N(t,X) = \inf_{\bx \in (\R^d)^N, \, \Pi \in P_N} \Big\{ V^N(t,\bx) + 2C_0 \E\big[ |X - Y_{\bx, \Pi}^N| \big] \Big\}.
    \end{align}
    Moreover, if $\ov{\bx}$ is a minimizer for the problem appearing in \eqref{def.tildeVn}, then it is straightforward to check that there exists $\ov{\Pi} \in P_N$ such that $(X, Y_{\ov{\bx}, \ov{\Pi}}^N)$ is an optimal coupling of $\cL(X)$ and $m_{\bar{\bx}}^N$, and hence there exists at least one optimizer $(\ov{\bx}, \ov{\Pi})$ for the problem appearing in \eqref{hatVn.rep}. For the remainder of the proof, we fix $\ov{\Pi}$ and write $Y_{\bx}^N = Y_{\bx, \ov{\Pi}}^N$ for simplicity. 
    
    Now suppose that $(t,X,a,p,A) \in (0,T) \times \mcl H \times \RR \times \mcl H \times \sym(\mcl H)$ is such that $(a, p, A) \in J^{2,+}( - \hat{V}^N \big)(t,X) = -J^{2,-} \hat{V}^N(t,X)$. It follows that there exists $\Phi \in C^{1,2}([0,T) \times \cH)$ such that 
    \begin{align*}
        &\partial_t \Phi(t,X) = -a, \quad \nabla \Phi(t,X) = -p, \quad \nabla^2 \Phi(t,X) = -A, 
        \\
        & \hat{V}^N(t,X) - \Phi(t,X) \leq \hat{V}^N(s,Y) - \Phi(s,Y), \quad \forall \,\, (s,Y) \in [0,T) \times \cH. 
    \end{align*}
    Thus for any $s \in [0,T]$, $Y \in \cH$, $\bx \in (\R^d)^N$, we have 
    \begin{align*}
        V^N(t,\ov \bx) &+ 2 C_0 \E\big[ |X - Y_{ \ov \bx}^N| \big] - \Phi(t,X) = \hat{V}^N(t,X) - \Phi(t,X) \leq \hat{V}^N(s,Y) - \Phi(s,Y)  
        \\
        &\leq V^N(s,\bx) + 2C_0 \E\big[ |Y - Y_{\bx}^N| \big] - \Phi(s,Y).
    \end{align*}
    Choosing $Y = X - Y_{\ov \bx}^N + Y_{  \bx}^N$, this becomes 
    \begin{align*}
        V^N(t,\ov \bx) &+ 2 C_0 \E\big[ |X - Y_{ \ov \bx}^N| \big] - \Phi(t,X)\leq V^N(s,\bx) + 2 C_0 \E\big[ |X - Y_{ \ov \bx}^N| \big] - \Phi(s,X - Y_{\ov \bx}^N + Y_{\bx}^N), 
    \end{align*}
    and so, for each $s \in [0,T)$, $\bx \in (\R^d)^N$,
    \begin{align*}
        V^N(t,\ov{\bx}) - \Psi^N(t,\ov \bx) \leq V^N(s,\bx) - \Psi^N(s,\bx)
    \end{align*}
    where $\Psi^N : [0,T) \times (\R^d)^N \to \R$ is defined as $\Psi^N(s,\bx) = \Phi\big(s, X - Y_{\ov \bx}^N + Y_{\bx}^N\big)$. We deduce from the viscosity supersolution property of $V^N$ and Lemma \ref{lem.projection.deriv} that 
    \begin{align*}
       &  a + \kappa \sum_{i = 1}^d A \big[ e_i, e_i ] + \hat{H}\Big( Y_{\ov{\bx}}^N, \E\big[-p | \cF_{\ov{\Pi}} \big] \Big) 
        \\
        &\quad = - \partial_t \Psi^N(t, \ov{\bx}) - \kappa \sum_{i,j = 1}^N \tr\big(D_{x^ix^i} \Psi^N(t,\ov{\bx}) \big) + \frac{1}{N} \sum_{i = 1}^N H\Big( \ov{x}^i, D_{x^i} \Psi^N(t,\ov{\bx}), m_{\ov{\bx}}^N \Big) \geq 0, 
    \end{align*}
    from which it follows, thanks to the growth assumption \eqref{ass:quadraticgrowthH} on $H$, that
    \begin{align*}
        a \geq  - \kappa d \|A\|_{\op} - C_H\big(1 + \|p\|_2^2\big),
    \end{align*}
    and thust the condition \eqref{BT} is satisfied.
\end{proof}

Given parameters $\eps, \eta, \lambda, \gamma > 0$, $N \in \N$, we again let $(\ov{t}, \ov{X}, \ov{Y})$ denote a maximum point of $\hat{\Phi}_{\eps, \eta, \lambda, \gamma}^N$. We also introduce parameters $\beta, \delta > 0$, and a partition $\Pi \in P_N$ of $\Omega$, and stress that the proportionality constants below are independent of these three new parameters as well.

The main goal of this section is to prove the following result, which shows that if $\lambda$ is large, then $\ov{t}$ must be close to $T$. Recall the definition of $\rho_N$ from \eqref{defn:rhoN} in the statement below.

\begin{prop}
    \label{prop.lambabound}  
    There exists a constant $C_0 > 0$ such that, for any $\eps, \eta, \lambda, \gamma, \beta, \delta > 0$, $N \in \N$, maximum point $(\ov{t}, \ov{X}, \ov{Y})$ of $\hat{\Phi}_{\eps, \eta, \lambda, \gamma}^N$, and $\Pi \in P_N$, the following holds: if 
    \begin{align*}
        \lambda \geq C_0 \bigg( \eta + \delta\Big( 1 +  \frac{\kappa}{\beta} \Big) + \eps +  \frac{\beta}{\eps} + \frac{1}{\delta \eps^2} \rho_N^2(\ov Y, \Pi) + \frac{1}{\delta \eps } \rho_N(\ov Y, \Pi) + \rho_N\Big( \frac{\ov{X}- \ov{Y}}{\eps}, \Pi \Big) \bigg),
    \end{align*}
    then 
    \begin{align*}
        T - \ov{t} \leq C_0 \Big( \beta  + \frac{1}{\delta \eps} \rho_N(\ov Y, \Pi)^2 + \frac{1}{\delta} \rho_N(\ov{Y}, \Pi) \Big).
    \end{align*}
\end{prop}

\begin{proof}
We denote by $\cH_0$ the finite-dimensional subspace of $\cH$ spanned by the constant random vectors $(e^i)_{i = 1,\ldots,d}$. We define $P_{\cH_0}  : \cH \to \cH$ to be the orthogonal projection onto $\cH_0$, and set $P_{\cH_0^{\perp}} = \text{Id}_{\cH} - P_{\cH_0}$. 

Let $(\ov{t}, \ov{X}, \ov{Y})$ be a maximum point of $\hat{\Phi}_{\eps, \eta, \lambda,\gamma}^N$, and suppose that $\ov{t} < T$. Then, by Lemma \ref{lem.hatvn.cks}, both
\[
    (t, X) \mapsto \hat U(t,X) - \frac{\eta}{2} (1 + T - t) \norm{X}_2^2
\]
and $- \hat V$ satisfy condition \eqref{BT}. We can therefore apply Lemma \ref{lem.cks} with $H_1 = H_2 = \mcl H$, $Z_1 = Z_2 = \mcl H_0$, 
\[
    a = -\lambda - \gamma \ov t^{-2},
\]
\[
    (p_1,p_2) = \left( \frac{\ov X - \ov Y}{\eps} , - \frac{\ov X - \ov Y}{\eps} \right),
\]
and
\[
    \mcl A = 
    \frac{1}{\eps}
    \begin{pmatrix}
        \Id_{\mcl H} & - \Id_{\mcl H} \\
        - \Id_{\mcl H} & \Id_{\mcl H}.
    \end{pmatrix}
\]
We will take $\eps = \kappa$ in the application of Lemma \ref{lem.cks}, where $\eps$ agrees with the parameter used throughout this section; observe then that
\[
    \mcl B := \mcl A + \eps \mcl A^2 = \frac{3}{\eps} 
    \begin{pmatrix}
        \Id_{\mcl H} & - \Id_{\mcl H} \\
        - \Id_{\mcl H} & \Id_{\mcl H}
    \end{pmatrix},
    \quad
    \mcl B + \eps \mcl B^2 = \frac{19}{\eps} 
    \begin{pmatrix}
        \Id_{\mcl H} & - \Id_{\mcl H} \\
        - \Id_{\mcl H} & \Id_{\mcl H}
    \end{pmatrix},
    \quad \text{and} \quad
    \norm{\mcl B}_\op = \frac{6}{\eps}.
\]
Invoking Lemma \ref{lem.cks} and using elementary properties of sub/superjets, we infer that there exist $a_1, a_2 \in \R$, $A, B \in \text{Sym}(\cH)$, such that
\begin{align*}
    &A = P_{\cH_0} A P_{\cH_0}, \quad B = P_{\cH_0} B P_{\cH_0}, \\
    &\Big(a_1 - \frac{\eta}{2} \norm{\ov X}_2^2, \frac{\ov{X} - \ov{Y}}{\eps} + \eta\big(1 + T - \ov{t} \big) \ov{X}, A + \eta(1 + T- \ov{t}) \text{Id}_{\cH} + \frac{3}{\eps} P_{\cH_0^{\perp}} \Big) \in \ov{J}^{2,+} \hat{U}(\ov{t}, \ov{X}), 
    \\
    &\Big(-a_2, \frac{\ov{X} - \ov{Y}}{\eps}, B - \frac{3}{\eps} P_{\cH_0^{\perp}}  \Big) \in \ov{J}^{2,-} \hat{V}^N(\ov{t}, \ov{Y}), 
    \\
    &a_1 - a_2 =  - \lambda - \gamma \ov{t}^{-2} \leq  - \lambda , \quad \text{and}
    \\
    &- \frac{7}{\eps} \text{Id}_{\cH \times \cH} \leq \begin{pmatrix}
        A & 0 \\
        0 & - B
    \end{pmatrix}
    \leq \frac{19}{\eps}  \begin{pmatrix}
      P_{\cH_0}  &
         -P_{\cH_0} \\
         -P_{\cH_0} & P_{\cH_0}
    \end{pmatrix}. 
\end{align*}
Now we apply the definition of the limiting superjet $\ov{J}^{2,-}(\ov{t},\ov{Y})$, to find a sequence of points
\[
	\left(t_n,Y_n, q_n, p_n, R_n \right) \in [0,T] \times \mcl H \times \R \times \mcl H \times \text{Sym}(\mcl H), \quad n \in \N
\]
such that 
\[
	\lim_{n \to \infty} (t_n, Y_n, q_n, p_n, R_n) = \left(\ov t,  \ov Y, a_2,  \frac{ \ov X -  \ov Y}{\eps}, B - \frac{3}{\eps} P_{\cH_0^{\perp}} \right)
\]
and
\[
	[t_n, T] \times \cH \ni (t,Y) \mapsto \hat{V}^N(t,Y) - F_n(t,Y)
\]
attains a local minimum at $(t_n,Y_n)$, where
\[
	F_n(t,Y) := (q_n- n^{-1})(t-t_n) +  \left\langle p_n, Y - Y_n \right \rangle + \frac{1}{2}(R_n - n^{-1} \Id)[ Y - Y_n, Y - Y_n].
\]
We note for later use that 
\begin{align} \label{phin.derivatives}
    \nabla \phi_n(t,Y) = p_n + (R_n - n^{-1} \text{Id}_{\cH})(Y - Y_n), \quad \nabla^2 \phi_n(t,Y) = (R_n - n^{-1} \text{Id}_{\cH}).
\end{align}
Fix $\beta > 0$ and define
\[
   h_\beta(r) = (r^2 + \beta^2)^{1/2}, \quad r \in \R.
\]
Observe that 
\begin{equation}\label{h_beta_bounds}
    r \vee \beta \le h_\beta(r) \le r + \beta, \quad   |h_\beta'| \le 1, \quad \text{and} \quad 0 \le h''_\beta \le \beta^{-1}.
\end{equation}
Now fix an $N$-partition $\Pi = (\Omega_i)_{i=1}^N \in P_N$ of $\Omega$, and, for some $\delta > 0$, we consider the function
\begin{equation}\label{doubling.2}
	[t_n, T] \ni (r,\bx) \mapsto V^N(r,\bx) - F_n(r,Y^N_\bx) + \delta h_\beta( \norm{Y^N_\bx - Y_n}_2 ) + \delta (r-t_n),
\end{equation}
where above and for the rest of the argument we write $Y^N_\bx = Y^N_{\bx, \Pi}$. 

We now claim that the function \eqref{doubling.2} attains a local minimum at some $(r_n, \bx_n) \in [t_n,T] \times (\Rd)^N$, which satisfies
\begin{equation}\label{upperbound.littlen}
	\limsup_{n \to \infty} \left( \norm{Y^N_{ \bx_n} - Y_n}_2 + |r_n - t_n| \right) \lesssim \beta  + \frac{1}{\delta \eps} \rho_N(\ov Y, \Pi)^2 + \frac{1}{\delta} \rho_N(\ov{Y}, \Pi), 
\end{equation}
with implied constant independent of $\eps, \eta, \lambda, \gamma, M, \beta, \delta, \Pi$. To see this, we choose $R > 2\rho_N(Y_n, \Pi)$ and define the compact set
\[
    \mcl C_R = \{ \by \in (\Rd)^N : \norm{Y^N_\by - Y_n}_{2} \le R \}.
\]
Note that $\mcl C_R$ is nonempty by the definition of $\rho_N(Y_n, \Pi)$.

Let $(r_n, \bx_n)$  be a point at which  \eqref{doubling.2} attains a minimum over $[t_n,T] \times \mcl C_R$. Then, for any $\by \in \mcl C_R$,
\begin{align*}
	\hat{V}^N(t_n,Y_n) &- F_n(t_n,Y_n) + \delta h_\beta\Big(\norm{Y^N_{ \bx_n} - Y_n}_2 \Big) + \delta(r_n -t_n)\\
	&\le \hat{V}^N(r_n,Y^N_{ \bx_n}) - F_n( r_n , Y^N_{\bx_n}) + \delta h_\beta \Big( \norm{Y^N_{ \bx^*_n} - Y_n}_2 \Big) + \delta (r_n -t_n)\\
	&\le V^N(t_n,\by) - F_n(t_n,Y^N_\by) + \delta h_\beta(\norm{Y^N_{ \by} - Y_n}_2).
\end{align*}
Rearranging terms gives
\begin{align*}
	&\delta h_\beta\left(\norm{Y^N_{ \bx_n} - Y_n}_2\right) + \delta |r_n-t_n| \\
    &\le \hat{V}^N(t_n,Y^N_\by) - \hat{V}^N(t_n,Y_n) + F_n(t_n,Y_n) - F_n(t_n,Y^N_\by) + \delta  h_\beta(\norm{Y^N_\by - Y_n}_2)  \\
	&\lesssim \norm{Y^N_\by - Y_n}_2  +  \frac{1}{\eps} \norm{Y^N_\by - Y_n}_2^2 + \delta h_\beta(\norm{Y^N_\by - Y_n}_2).
\end{align*}
From the assumption on $R$, there exists $\by \in \mcl C_R$ such that $\norm{Y^N_\by - Y_n}_2 \le 2 \rho_N(Y_n, \Pi)$. Upon choosing this $\by$ and dividing by $\delta >0$, we find that
\[
    \norm{Y^N_{\bx_n} - Y_n}_2 + |r_n - t_n| \lesssim 
    \beta  + \frac{1}{\delta \eps} \rho_N(Y_n, \Pi)^2 + \frac{1}{\delta} \rho_N(Y_n, \Pi).
\]
We then take $R$ larger than an appropriate multiple of the right-hand side above, which implies that $\bx_n$ is an interior minimum. The limiting statement then follows from the fact that $\mcl H \ni Y \mapsto \rho_N(Y, \Pi)$ is continuous and $\lim_{n \to \infty} \norm{Y_n - \ov{Y}}_{\mcl H} = 0$. Thus \eqref{upperbound.littlen} holds.  

From \eqref{upperbound.littlen}, we deduce that there is a constant $C_0$ such that, if 
\begin{align} \label{ovt.bound}
    T - \ov{t} > C_0 \Big( \beta + \frac{1}{\delta \eps} \rho_N^2(\ov Y, \Pi) + \frac{1}{\delta} \rho_N(\ov Y, \Pi)\Big), 
\end{align}
then
\begin{align} \label{rn.neqt}
    r_n < T \quad \text{for all large enough $n$.}
\end{align}
We now suppose that $\ov{t}$ satisfies \eqref{ovt.bound}, and aim to bound $\lambda$ from above. We first use the fact that $\hat{U}$ is a subsolution of \eqref{hjb.hilbert}, and the fact that $P_{\cH_0^{\perp}} e_k = 0$ for $k = 1,\ldots,d$, to find that
\begin{align} \label{hatu.subsol}
    - a_1 + \frac{\eta}{2} \norm{\ov X}_2^2 - \eta (1 + T - \ov{t}) d - \kappa \sum_{k = 1}^d A [e_k, e_k] + \hat{H}\Big( \ov{X}, \frac{\ov{X} - \ov{Y}}{\eps} + \eta(1 + T - \ov{t}) \ov{X}\Big) \leq 0.
\end{align}
Next, 
define
$$f_n(r,\bx) := F_n(r,Y^N_\bx) - \delta h_\beta( \norm{Y^N_\bx - Y_n}_2 ) - \delta (r-t_n),$$
so that 
\[
   [t_n, T] \times (\R^d)^N \ni  (r, \bx) \mapsto V^N(r,\bx) - f_n(r,\bx)
\]
has a local minimum at $(r_n,\bx_n)$, and for all large enough $n$, we have $r_n < T$. Note that this minimum is local with  respect to $[t_n,T) \times (\R^d)^N$ rather than $(0,T) \times (\R^d)^N$. It turns out that the viscosity solution definition may still be applied to this setting even if it happens that $r_n = t_n$, which can be justified by adding an additional penalization of the form $\ov \gamma( t_n^{-1} - r^{-1})$ and sending $\ov \gamma \to 0$; we omit the details to ease the presentation, since this is standard technique in the finite-dimensional viscosity solution theory. 

The super-solution property for $V^N$ thus yields
\begin{align} \label{vn.supersol} 
    -\partial_t f_n(r_n,\bx_n) - \kappa \sum_{i,j=1}^N \tr(D_{x^{i}x^j}f_n(r_n,\bx_n)) + \frac{1}{N} \sum_{i=1}^N H(x^{i}_n, N D_{x^{i}}f_n(r_n,\bx_n),m^N_{\bx_n}) \geq 0.
\end{align}
Notice that by design $\partial_t f_n(r_n,\bx_n) = q_n - \frac{1}{n} - \delta$, while Lemma \ref{lem.projection.deriv} allows us to compute
\begin{align*}
    &D_{x^i} f_n(r_n,\bx_n) = \frac{1}{N} \E\big[ p_n | \Omega^i \big] + \frac{1}{N} \E\big[ (R_n - n^{-1} \text{Id}_{\cH}) (Y_{\bx_n} - Y_n) | \Omega^i \big] + \frac{\delta}{N} \E\Big[ \nabla \Big( h_{\beta}\big( \| \cdot - Y_n\|_2 \big) \Big)(Y_{\bx_n}) | \Omega^i \Big], 
    \\
    &\tr\big( D_{x^ix^j} f_n(r_n^*, \bx_n^*) \big) = \sum_{k = 1}^d R_n [1_{\Omega^i} e_k, 1_{\Omega^i} e_k] - \frac{d}{nN} 1_{i = j} - \delta \sum_{k = 1}^d \nabla^2 \Big( h_{\beta}\big( \| \cdot - Y_n\|_2 \big) \Big)(Y_{\bx_n}^N) \big[1_{\Omega^i}, 1_{\Omega^j} \big].
\end{align*}
We can thus rewrite \eqref{vn.supersol} as
\begin{align} \label{vn.supersol.2}
    &- q_n + \frac{1}{n} + \delta + \frac{\kappa d}{n} - \kappa \sum_{k = 1}^d R_n [e_k, e_k] - \kappa  \delta \sum_{k = 1}^d \nabla^2 \Big( h_{\beta}\big( \| \cdot - Y_n\|_2 \big) \Big)(Y_{\bx_n}^N) \big[e_k, e_k \big]
    \nonumber \\
    &\quad + \hat{H} \Big( Y_{\bx_n}^N, \E\Big[ p_n + (R_n - n^{-1} \text{Id}_{\cH}) (Y_{\bx_n^*} - Y_n) + \delta \nabla \Big( h_{\beta}\big( \| \cdot - Y_n\|_2 \big) \Big)(Y_{\bx_n})   | \Pi\Big]\Big) \geq 0.
\end{align}
Note that the bounds for $h_\beta$ in \eqref{h_beta_bounds} give
\begin{align*}
    \norm{ \nabla \Big( h_{\beta}\big( \| \cdot - Y_n\|_2 \big) \Big)(Y_{\bx_n})  }_2 \lesssim 1, \quad \norm{ \nabla^2 \Big( h_{\beta}\big( \| \cdot - Y_n\|_2 \big) \Big)(Y_{\bx_n}) }_{\op} \lesssim \beta^{-1}.
\end{align*}
Thus we can subtract \eqref{vn.supersol.2} from \eqref{hatu.subsol}, and use the Lipschitz regularity of $H$, to obtain
\begin{align}
    \label{errorterms}
    - a_1 + \frac{\eta}{2} \norm{\ov X}_2^2 + q_n &\lesssim \kappa \sum_{k = 1}^d (A - R_n ) [e_k,e_k] +  \eta(1 + \|\ov{X}\|_1 \big) + \frac{1}{n} + \delta + \frac{\kappa \delta}{\beta}
   \nonumber  \\
    &\qquad  + \frac{1}{\eps} \|Y_{\bx_n} - Y_n \|_{1} + \hat{H}\Big( Y_{\bx_n}^N, \E\big[ p_n | \Pi\big]\Big) - \hat{H}\Big(\ov{X}, \frac{\ov{X} - \ov{Y}}{\eps} \Big)
   \nonumber  \\
   &\lesssim \kappa \sum_{k = 1}^d (A - R_n ) [e_k,e_k] +  \eta(1 + \|\ov{X}\|_1 \big) + \frac{1}{n} + \delta + \frac{\kappa \delta}{\beta}
   \nonumber  \\
    &\qquad  + \frac{1}{\eps} \|Y_{\bx_n^*} - Y_n \|_{1} + \hat{H}\Big( Y_{\bx_n}^N, \E\big[ p_n | \Pi\big]\Big) - \hat{H}\Big(Y_{\bx_n}^N, \frac{\ov{X} - \ov{Y}}{\eps} \Big) + \|Y_{\bx_n}^N - \ov X \|_2
   \nonumber 
   \\
    &\lesssim \kappa \sum_{k = 1}^d (A - R_n ) [e_k,e_k] +  \eta(1 + \|\ov X\|_2 \big) + \frac{1}{n} + \delta  + \frac{\kappa \delta}{\beta}
   \nonumber  \\
    \nonumber &\qquad + \frac{1}{\eps} \|Y_{\bx_n} - Y_n \|_{2} + \|Y_{\bx_n} - \ov{X} \|_2 + \norm{ \E\big[ p_n | \Pi \big] - \frac{\ov{X} - \ov{Y}}{\eps}}_2
    \\
  \nonumber   &\lesssim \kappa \sum_{k = 1}^d (A - R_n ) [e_k,e_k] +  \eta(1 + \|\ov{X}\|_2 \big) + \frac{1}{n} + \delta  + \frac{\kappa \delta}{\beta}
    \\
    &\qquad + \frac{1}{\eps} \|Y_{\bx_n} - Y_n \|_{2} + \|Y_n - \ov X\|_2 +  \norm{ \E\big[ p_n | \Pi \big] - \frac{\ov{X} - \ov{Y}}{\eps}}_2. 
\end{align}
Since this holds for all large enough $n$, we can send $n \to \infty$ and use \eqref{upperbound.littlen} and the facts that $-a_1 + a_2 \ge \lambda$ and $\| \ov Y - \ov X\|_{2} \lesssim \eps$ to obtain 
\begin{align*}
    \lambda + \frac{\eta}{2} \norm{\ov{X}}_2^2 &\lesssim \sum_{k = 1}^d \big(A - B \big)[e_k,e_k] + \eta(1 + \|\ov{X}\|_2) + \delta\Big( 1 + \frac{\kappa}{\beta} \Big) + \eps
    \\
    &\quad + \frac{\beta}{\eps} + \frac{1}{\delta \eps^2} \rho_N^2(Y^*, \Pi) + \frac{1}{\delta \eps } \rho_N(Y^*, \Pi) + \norm{\E\Big[ \frac{\ov{X} - \ov{Y}}{\eps} | \Pi \Big] - \frac{\ov{X} - \ov{Y}}{\eps} \Big]}_2
    \\
    &\lesssim \eta(1 + \norm{\ov{X}}_2) + \delta\Big( 1 +  \frac{\kappa}{\beta} \Big) + \eps + \frac{\beta}{\eps} + \frac{1}{\delta \eps^2} \rho_N^2(Y^*, \Pi) + \frac{1}{\delta \eps } \rho_N(Y^*, \Pi) + \rho_N\Big( \frac{\ov{X}- \ov{Y}}{\eps}, \Pi \Big). 
\end{align*}
An application of Young's inequality to the first term on the right-hand side completes the proof. 

\end{proof}

\subsubsection{The choice of $\Pi$ and the end of the proof}

We are now ready to complete the proof of Theorem \ref{thm.main} in the case $d \geq 2$. This is done by appropriately choosing the partition $\Pi$ with the use of the simultaneous quantization result Proposition \ref{prop.quantization.sim}.

\begin{proof}[Proof of Theorem \ref{thm.main}]
  We continue to use the same notation from the previous sections: $\eps, \eta, \lambda, \gamma > 0$ are parameters, and $(\ov{t}, \ov{X}, \ov{Y})$ denotes a maximum point of $\Phi_{\eps, \eta, \lambda,\gamma}^N$. We another introduce another parameter, based on the choice of the partition $\Pi$. In particular, by Proposition \ref{prop.quantization.sim} and Lemma \ref{lem.fistpenalization}, there is a constant $C$ such that for each $\eps, \eta, \lambda$ and $\alpha \in (0,1)$, we can find $\Pi$ such that
  \begin{align}
  \label{eq:23/09:14:41}
      \rho_N(\ov{Y}, \Pi) \leq C \eta^{-1} r_{\lfloor N^{\alpha} \rfloor , d} , \quad \text{and} \quad
      \rho_N\Big(\frac{\ov{X} - \ov{Y}}{\eps}, \Pi \Big) \leq C r_{\lfloor N^{1-\alpha} \rfloor , d}.
  \end{align}
  Combining this with Proposition \ref{prop.lambabound}, we see that there is a constant $C$ such that for each $\eps, \eta, \lambda, \gamma,  \beta, \delta,\alpha$, we have the implication
  \begin{align*}
      &\lambda \geq C\bigg( \eta + \delta\Big( 1 +  \frac{\kappa}{\beta} \Big) + \eps +  \frac{\beta}{\eps} + \frac{1}{\delta \eps^2 \eta^2}  r_{\lfloor N^{\alpha} \rfloor,d}^2 + \frac{1}{\delta \eps \eta } r_{\lfloor N^{\alpha} \rfloor ,d} + r_{\lfloor N^{1 -\alpha} \rfloor ,d} \bigg)
      \\
      &\qquad \qquad \implies 
        T - \ov{t} \leq C \Big( \beta  + \frac{1}{\delta \eps \eta^2} r^{2}_{ \lfloor N^{\alpha} \rfloor ,d} + \frac{1}{\delta \eta } r_{ \lfloor N^{\alpha} \rfloor ,d} \Big).
    \end{align*}
    In particular, taking 
    \begin{align*}
        \lambda = C\bigg( \eta + \delta\Big( 1 +  \frac{\kappa}{\beta} \Big) + \eps +  \frac{\beta}{\eps} + \frac{1}{\delta \eps^2 \eta^2}  r_{\lfloor N^{\alpha} \rfloor,d}^2 + \frac{1}{\delta \eps \eta } r_{\lfloor N^{\alpha} \rfloor ,d} + r_{\lfloor N^{1 -\alpha} \rfloor ,d} \bigg), 
    \end{align*}
    we can use \eqref{startingest} and Lemma \ref{lem.shorttime.higherdim} to get that for all $\eps, \eta, \beta, \delta, \gamma > 0$ and $\alpha \in (0,1)$,
    \begin{align} \label{beforekappa}
        &U(t,m_{\bx}^N) - V^N(t,\bx) \leq C \Big( (T - \ov{t}) + \eps + \eps^{-1} \eta^{-1} r_{\lfloor N^{\alpha} \rfloor,d} + \eta M_2(m_{\bx}^N) + \lambda + \gamma t^{-1} \Big)
       \nonumber  \\
        &\,\, \leq C \bigg( \eta + \delta\Big( 1 +  \frac{\kappa}{\beta} \Big) + \eps +  \frac{\beta}{\eps} + \frac{1}{\delta \eps^2 \eta^2} r^2_{\lfloor N^{\alpha} \rfloor,d} +  \frac{1}{\delta \eps \eta } r_{\lfloor N^{\alpha} \rfloor,d} +r_{\lfloor N^{1-\alpha} \rfloor,d}  + \eta M_2(m_{\bx}^N) + \gamma t^{-1}  \bigg). 
    \end{align}
    Sending $\gamma \to 0$ first and then choosing
\begin{align*}
  \eta = \eps, \quad  \beta = \eps^2, \quad \delta = \eps^3, \quad \eps = r_{\lfloor N^{\alpha} \rfloor, d}^{1/6}
\end{align*}
we find that for each $\alpha \in (0,1)$,
\begin{align*}
    U(t,m_{\bx}^N)& - V^N(t,\bx) \leq C \Big(r_{\lfloor  N^{\alpha} \rfloor , d}^{1/6} + r_{\lfloor  N^{1 - \alpha} \rfloor , d} + r_{\lfloor  N^{\alpha} \rfloor , d}^{1/6} M_2(m_{\bx}^N) \Big).
\end{align*}
Recalling the definition of $r_{N,d}$ in \eqref{def.rnd} and using the fact that $\log(1+N^\tau) \le \log(1+N)$ for all $\tau \in (0,1)$,  we estimate
\begin{align*}
    r_{\lfloor  N^{\alpha} \rfloor , d}^{1/6}
    \lesssim 
    \begin{cases}
       N^{-\frac{\alpha}{12}} \log\big(1 + N \big)^{1/12}, &  d = 2,
       \\
       N^{- \frac{\alpha}{6d}}, & d \geq 3,
    \end{cases}
\end{align*}
and
\begin{align*}
    r_{\lfloor  N^{1 - \alpha} \rfloor , d}
    \lesssim
    \begin{cases}
        N^{-\frac{1 - \alpha}{2}} \log\big( 1 + N \big)^{1/2}, & d = 2,\\
        N^{-\frac{1-\alpha}{d}}, & d \geq 3,
    \end{cases}
\end{align*}
and thus choose (when $N$ is sufficiently large relative to a universal constant)
\[
    \alpha =
    \begin{dcases}
        \frac{6}{7} \left( 1 - \frac{5}{6} \frac{\log \log (1+N)}{\log (1+N)} \right), & d = 2, \\
        \frac{6}{7}, & d \ge 3
    \end{dcases}
\]
to obtain
\begin{align*}
U(t,m_{\bx}^N)& - V^N(t,\bx) \lesssim \big(1 + M_2(m_{\bx}^N)\big) 
\begin{cases}
      N^{- \frac{1}{14}} \log\big(1 + N \big)^{1/7} &  d = 2,
       \\
      N^{-\frac{1}{7d}} & d \geq 3.
    \end{cases}
\end{align*}
In the case that $\kappa = 0$, we first send $\beta \to 0$ in \eqref{beforekappa} to obtain
\begin{align*}
    U(t,m_{\bx}^N) - V^N(t,\bx) \leq C\Big(\eta + \delta + \eps + \frac{1}{\delta \eps^2 \eta^2} r^2_{\lfloor N^{\alpha} \rfloor,d} + \frac{1}{\delta \eps \eta} r_{\lfloor N^\alpha \rfloor, d} + r_{\lfloor N^{1-\alpha} \rfloor, d} + \eta M_2(m_{\bx}^N) \Big), 
\end{align*}
and then choose $\eta = \delta = \eps = r_{\lfloor N^\alpha \rfloor, d}^{1/4}$ to find
\begin{align*}
    U(t,m_{\bx}^N) - V^N(t,\bx) \lesssim \big(1 + M_2(m_{\bx}^N) \big)\Big( r_{\lfloor N^\alpha \rfloor,d}^{1/4} + r_{\lfloor N^{1-\alpha} \rfloor, d} \Big), 
\end{align*}
and again optimizing in $\alpha$ gives 
\begin{align*}
U(t,m_{\bx}^N)& - V^N(t,\bx) \leq  C\big(1 + M_2(m_{\bx}^N)\big) \begin{cases}
      N^{- \frac{1}{10}} \log\big(1 + N \big)^{1/5} &  d = 2
       \\
      N^{-\frac{1}{5d}} & d \geq 3.
    \end{cases}
\end{align*}

\section{Toward improved estimates under stronger assumptions} \label{sec.improvements}

We now discuss, without presenting too many details, various additional assumptions or adjusted settings that can simplify the arguments and/or improve slightly the rates in Theorem \ref{thm.main}. 

\subsection{Bounded state space or stronger moments} First, to avoid compactness issues, we could assume that the data is periodic; this would mean that 
   \begin{align} \label{def.periodic}
    H(x,p,m) = \wt{H}\big(\pi(x), p, \pi_{\#} m \big), \quad G(m) = \wt{G}\big(\pi_{\#} m \big), 
   \end{align}
   where $\pi : \R^d \to \T^d \coloneqq \R^d / \Z^d$ is the quotient map from $\R^d$ to the $d$-dimensional flat torus $\T^d$, and 
   \begin{align} \label{def.wtH}
       \wt{H} : \T^d \times \R^d \times \cP(\T^d), \quad G : \cP(\T^d) \to \R
   \end{align}
   satisfying the conditions 
   \begin{align} \label{periodic.1}
       &|\wt H(x,p,m)|\leq C_H \big(1 + |p|^2 \big), 
       \\ \label{periodic.2}
       &|\wt H(x,p,m) - \wt H(x',p',m')| \leq C_H \big(1 + |p| \big)\Big( |x -x'| + |p-p'| + \bd_{\T^d,1} (m,m') \Big), 
       \\ \label{periodic.3}
       &|G(m)| \leq C_G, \quad |G(m) - G(n)| \leq C_G \bd_{\T^d,1}(m,n)
    \end{align}
    where we have abused notation by using $|x-y|$ for the usual flat distance between $x$ and $y$ on $\T^d$, and $\bd_{\T^d,1}$ denotes the 1-Wasserstein distance on $\T^d$. In this case, the parameter $\eta$ appearing in the doubling of variables arguments below plays no role. For example, in the case $d = 1$ we can replace the function $\Psi_{\eps, \eta, \lambda, \gamma}^N$ by 
    \begin{align*}
        \Psi_{\eps, \lambda, \gamma}^N(t,m,\bx) = U(t,m) - V^N(t,\bx) - \frac{1}{2\eps} \bd_2^2(m,m_{\bx}^N) - \lambda(T-t) - \gamma \Big(\frac{1}{t} - \frac{1}{T} \Big), 
    \end{align*}
    and argue that by periodicity, we can find a minimizer $(\ov t, \ov m, \ov{\bx})$ of $\Psi_{\eps, \lambda, \gamma}^N$ with $\ov m$ and $m_{\ov \bx}^N$ supported on, say, $[-2,2]$. Similarly, when $d > 1$, we can replace the function $\Phi_{\eps, \eta, \lambda, \gamma}^N$ by 
    \begin{align*}
        \Phi_{\eps, \lambda, \gamma}^N(t,m,n) = U(t,m) - \wt{V}^N(t,n) - \frac{1}{2\eps} \bd_2^2(m,n) - \lambda(T-t) - \gamma \Big(\frac{1}{t} - \frac{1}{T} \Big), 
    \end{align*} 
    and the function $\hat{\Phi}_{\eps, \lambda, \gamma}^N$ by 
    \begin{align*}
        \hat{\Phi}_{\eps, \lambda, \gamma}^N(t,X,Y) = \hat{U}(t,X) - \hat{V}^N(t,Y) - \frac{1}{2\eps} \|X - Y\|_2^2 - \lambda(T-t) - \gamma \Big(\frac{1}{t} - \frac{1}{T} \Big), 
    \end{align*}
    and argue by periodicity that we can find a maximizer $(\ov t, \ov X, \ov Y)$ of $\hat{\Phi}_{\eps, \lambda, \gamma}^N$ with $X,Y \in [-2,2]^d$ almost surely. From here, mimicking the proofs of Theorem \ref{thm.main} reveals that if \eqref{def.periodic} holds with $\wt{G}$, $\wt{H}$ satisfying \eqref{def.wtH}, \eqref{periodic.1}, \eqref{periodic.2}, and \eqref{periodic.3} hold, then \eqref{rates.common} and \eqref{rates.nocommon} can be improved to 
    \begin{align} \label{rates.common.periodic}
       | V^N(t,\bx) - U(t,m_{\bx}^N) | \leq C  \times \begin{cases}
           N^{-\frac{1}{2}} & d = 1, 
           \\
           N^{- \frac{1}{12}} \log(1 + N)^{1/6} & d = 2
           \\
           N^{- \frac{1}{6d}} & d \geq 3
       \end{cases}
    \end{align}
   when $\kappa > 0$, and to 
    \begin{align} \label{rates.nocommon.periodic}
       | V^N(t,\bx) - U(t,m_{\bx}^N) | \leq C   \times \begin{cases}
           N^{-\frac{1}{2}} & d = 1, 
           \\
           N^{- \frac{1}{8}} \log(1 + N)^{1/4} & d = 2
           \\
           N^{- \frac{1}{4d}} & d \geq 3.
       \end{cases}
    \end{align}
    when $\kappa = 0$.

    In the general, non-periodic setting, one can work toward the same exponents at the expense of controlling stronger moments of the form
    \[
        M_q(m^N_\bx) = \frac{1}{N} \sum_{i = 1}^N |x^i|^q 
    \]
    for some sufficiently large $q > 2$. This is achieved by replacing the penalization $\eta \mcl M_2(m)$ in the doubling-of-variables argument with $\eta \mcl M_q(m)$; note that care must be taken since the map $X \mapsto \norm{X}_q^q$ becomes a singular test function on $\mcl H = L^2$. The main difference in the argument is in the use of the error estimates
    \[
        r_{N,d; q} := \sup\left\{ \inf_{\bx \in (\Rd)^N} \mbf d_q(\mu, \mu^N_\bx) : \mcl M_q(\mu) \le 1 \right\}
    \]
    (note that $q = \infty$ corresponds to the rates $r_{N,d}$ defined in \eqref{def.rnd} and used throughout the paper). The various moments $\norm{\ov X}_q$ and $\norm{\ov Y}_q$ of the extreme points then remain in the argument until the final application of Young's inequality, and the negative powers of $\eta$ make no appearance, similar to the compact case treated above.

    \subsection{Smoother data and stochastic cancellations} It turns out that, due to the averaged structure of the equation \eqref{hjb.inf}, increasing the regularity of the Hamiltonian $H$ beyond the Lipschitz threshold actually leads to a slight improvement in the exponent. For instance, if we add the assumption that, for each $R > 0$, there exists $C_R > 0$ such that 
    \begin{align} \label{assump.c11}
        \sup_{m,x} \| H(x, \cdot, m) \|_{C^{1,1}(B_R)} \leq C_R,
    \end{align}
    then, in this case, instead of using the Lipschitz  regularity of $H$ to bound the term
    \begin{align*}
        \hat{H}\big(Y_{\bx_n}^N, \E[ p_n | \cF_{\Pi}] \big) - \hat{H} \Big(Y_{\bx_n}^N, \frac{\ov X - \ov Y}{\eps} \Big) = \E\Big[H\big(Y_{\bx_n}^N, \E[ p_n | \cF_{\Pi}] m_{\bx_n}^N \big) - H \Big(Y_{\bx_n}^N, \frac{\ov X - \ov Y}{\eps} , m_{\bx_n}^N \Big) \Big]
    \end{align*} 
    in the proof of Proposition \ref{prop.lambabound}, we can instead use the $C^{1,1}$ regularity in $p$ to argue as follows:
\begin{align*}
  H\Big( &Y^N_{\bx_n}, \E[p_n | \cF_{\Pi}] , m_{\bx_n}^N \Big) - H\Big( Y^N_{\oline \bx}, \frac{\oline X - \ov{Y}}{\eps} , m_{\bx_n}^N \Big)
  \\
  &\leq H\Big( Y^N_{\bx_n}, \E\Big[ \frac{\oline X - \ov Y }{\eps} \mid \cF_{\Pi} \Big], m_{\bx_n}^N \Big) - H\Big( Y^N_{\oline \bx}, \frac{\oline X - \ov Y}{\eps} , m_{\bx_n}^N \Big) + \norm{p_n - \frac{\ov X - \ov Y}{\eps}}_2
   \\
   &\leq D_pH \Big( Y^N_{\bx_n}, \E\Big[ \frac{\oline X - \ov Y}{\eps} \mid \cF_{\Pi} \Big], m_{\bx_n}^N \Big) \cdot \Big( \E\Big[ \frac{ \ov X - \ov Y}{\eps} | \cF_{\Pi} \Big] - \frac{ \ov X - \ov Y}{\eps} \Big)
   \\
   &\qquad + C \Big| \E\Big[ \frac{ \ov X - \ov Y}{\eps} | \cF_{\Pi} \Big] - \frac{ \ov X - \ov Y}{\eps} \Big|^2 +  \norm{p_n - \frac{\ov X - \ov Y}{\eps}}_2 
\end{align*}
The term involving $D_pH$ is $\mcl F_\Pi$-measurable, and therefore the entire first term on the right-hand side has mean zero. Taking the expectation, we therefore get
\begin{align*}
   H\Big( &Y^N_{\bx_n}, \E[p_n | \cF_{\Pi}] , m_{\bx_n}^N \Big) - H\Big( Y^N_{\oline \bx}, \frac{\oline X - \ov{Y}}{\eps} , m_{\bx_n}^N \Big) \leq \rho_N^2 \Big( \frac{\ov X - \ov Y}{\eps}, \Pi\Big) + \norm{p_n - \frac{\ov X - \ov Y}{\eps}}_2.
\end{align*}
Plugging this estimate into the proof of Proposition \ref{prop.lambabound}, we see that if \eqref{assump.c11} holds, then in the statement of Proposition \ref{prop.lambabound} we can replace $\rho_N \Big( \frac{\ov X - \ov Y}{\eps}, \Pi\Big)$ by $\rho_N^2 \Big( \frac{\ov X - \ov Y}{\eps}, \Pi\Big)$. Tracing this through the proof of Theorem \ref{thm.main}, we find that for $d \geq 2$, we can improve \eqref{rates.common} and \eqref{rates.nocommon} to
\begin{align} \label{rates.common.c11}
       | V^N(t,\bx) - U(t,m_{\bx}^N) | \leq C \Big(1 + \frac{1}{N} \sum_{i = 1}^N |x^i|^2 \Big) \times \begin{cases}
           N^{- \frac{1}{13}} \log(1 + N)^{2/13} & d = 2
           \\
           N^{- \frac{2}{13d}} & d \geq 3.
       \end{cases}
    \end{align}
   when $\kappa > 0$, and to 
    \begin{align} \label{rates.nocommon.c11}
       | V^N(t,\bx) - U(t,m_{\bx}^N) | \leq C \Big(1 + \frac{1}{N} \sum_{i = 1}^N |x^i|^2 \Big)  \times \begin{cases}
           N^{- \frac{1}{9}} \log(1 + N)^{2/9} & d = 2
           \\
           N^{- \frac{2}{9d}} & d \geq 3.
       \end{cases}
    \end{align}
    when $\kappa = 0$. 

    Finally, if we are in the periodic case \textit{and} we have $C^{1,1}$ regularity in $p$, i.e. \eqref{def.periodic}, \eqref{periodic.1}, \eqref{periodic.2}, \eqref{periodic.3}, and \eqref{assump.c11} all hold, then we can combine the arguments outlined above to obtain the improvements (for $d \geq 2$): 
    \begin{align} \label{rates.common.periodic.c11}
       | V^N(t,\bx) - U(t,m_{\bx}^N) | \leq C \times \begin{cases}
           N^{- \frac{1}{11}} \log(1 + N)^{2/11} & d = 2
           \\
           N^{- \frac{2}{11d}} & d \geq 3.
       \end{cases}
    \end{align}
   when $\kappa > 0$, and
    \begin{align} \label{rates.nocommon.periodic.c11}
       |V^N(t,\bx) - U(t,m_{\bx}^N) | \leq C  \times \begin{cases}
           N^{- \frac{1}{7}} \log(1 + N)^{2/7} & d = 2
           \\
           N^{- \frac{2}{7d}} & d \geq 3.
       \end{cases}
    \end{align}
    when $\kappa = 0$. 
\end{proof}

\appendix 
\section{Proof of the Lipschitz bounds}

This appendix is devoted to a proof of Proposition \ref{prop.uniformlip}.

\begin{proof}
    We begin by proving that there is a constant $C$ such that 
    \begin{align} \label{spatiallip}
        |V^N(t,\bx) - V^N(t,\by)| \leq \frac{C}{N} \sum_{i = 1}^N |x^i - y^i|
    \end{align}
    for all $t \in [0,T]$, $\bx,\by \in (\R^d)^N$. To do this, we introduce the function 
    \begin{align*}
        \lambda : [0,T] \to \R, \quad \lambda(t) = C_0 \exp\big(C_0(T-t) \big), \quad C_0 = (3 C_H) \vee C_G , 
    \end{align*}
    and for simplicity we introduce the notation 
    \begin{align*}
        \Phi : (\R^d)^N \to \R, \quad \Phi(\bz) = \frac{1}{N} \sum_{i = 1}^N |z^i|. 
    \end{align*}
    Our goal is to show that
    \begin{align}
        M_0 \coloneqq \sup_{t \in [0,T], \bx,\by \in (\R^d)^N} \Big\{V^N(t,\bx) - V^N(t,\by) - \lambda(t) \Phi(\bx - \by)\Big\} \leq 0, 
    \end{align}
    which clearly implies \eqref{spatiallip} (with $C = \lambda(0) = C_0 \exp(C_0 T)$). Suppose towards a contradiction that $M_0 > 0$. Then for all $\delta$ small enough, 
    \begin{align*}
        M_{\delta} \coloneqq \sup_{t \in [0,T], \bx,\by \in (\R^d)^N} \Big\{V^N(t,\bx) - V^N(t,\by) - \lambda(t) \Phi(\bx - \by) - \frac{\delta}{2} \sum_{i = 1}^N |x^i|^2 - \delta\Big(\frac{1}{t} - \frac{1}{T}\Big) \Big\} \geq \frac{M}{2} > 0.
    \end{align*}
    The comparison principle, and the continuity and boundedness of $G$, imply that $V^N$ is continuous and bounded, and so the expression defining $M_{\delta}$ has at least one optimizer $(t_{\delta}, s_{\delta}, \bx_{\delta}, \by_{\delta})$, and because $M_{\delta} > 0$ and $C_0 \geq C_G$, there is a constant $\gamma = \gamma(\delta)$ depending on $\delta$ such that any optimizer satisfies $t_{\delta} \in (\gamma,  T - \gamma)$. Now for $\eps > 0$, we define 
    \begin{align*}
        \Phi_{\eps} : (\R^d)^N \to \R, \quad \Phi_{\eps}(\bz) = \frac{1}{N} \sum_{i = 1}^N \big(|z^i|^2 + \eps^2\big)^{1/2}.
    \end{align*}
    Since $M_{\delta} > 0$ (for $\delta$ small enough) it follows that for $\eps$ small enough (depending on $\delta$) we have 
    \begin{align} \label{M.epsdelta}
        &M_{\eps, \delta} \coloneqq 
        \nonumber \\
        &\sup_{t \in [0,T], \bx, \by \in  (\R^d)^N} \Big\{ V^N(t,\bx) - V^N(s,\by)- \lambda(t) \Psi_{\eps}(\bx - \by) - \frac{\delta}{2} \sum_{i = 1}^N |x^i|^2 - \delta\Big(\frac{1}{t} - \frac{1}{T}\Big) \Big\} > 0.
    \end{align}
    Moreover, standard arguments show that if $(t_{\delta, \eps}, s_{\delta,\eps}, \bx_{\delta, \eps}, \by_{\delta, \eps})$ is any optimizer for this expression, then any limit point as $\eps \downarrow 0$ takes the form $(t_{\delta}, t_{\delta}, \bx_{\delta}, \bx_{\delta})$, where $(t_{\delta},\bx_{\delta})$ is an optimizer for the problem defining $M_{\delta}$. In particular, for $\eps$ small enough (depending on $\delta$), we must have $0 < t_{\delta, \eps} < T$, $0 < s_{\delta, \eps} < T$.
    
    We next use the equation for $V^N$. For this, we note that $\Phi_{\eps}(\bx,\by) = \Psi_{\eps}(\bx - \by)$ satisfies
    \begin{align*}
        &D_{x^i} \Phi_{\eps}(\bx,\by) = - D_{y^i} \Phi(\bx,\by) = \frac{x^i - y^i}{\big( |x^i - y^i|^2 + \eps^2 \big)^{1/2}}, 
        \\
        &D^2_{\bx,\by} \Phi_{\eps}(\bx,\by) = \begin{pmatrix}
            D^2 \Psi_{\eps}(\bx - \by) & - D^2 \Psi_{\eps}(\bx - \by) 
            \\
            - D^2 \Psi_{\eps}(\bx - \by) & D^2 \Psi_{\eps}(\bx - \by)
        \end{pmatrix} \in \R^{2dN \times 2dN} \cong (\R^{d \times d})^{2N \times 2N}. 
    \end{align*}
    The matrix $D^2 \Phi_{\eps} \in \R^{dN \times dN}$ can also be computed explicitly, but its exact form will not be used. The parabolic form of Ishii's lemma (see for instance \cite[Theorem 9]{Crandall_Ishii_max}) implies that we can find real numbers $a, b$ and matrices $\bm{A} = (A^{i,j})_{i,j = 1,\ldots,N}, \bm{B} = (B^{i,j})_{i,j = 1,\ldots,N} \in (\R^d)^N \cong \R^{dN \times dN}$ such that 
    \begin{align*}
        &a - b = \lambda'(t_0) \Phi_{\eps}(\bx_{\delta, \eps}, \by_{\delta, \eps}) + t_{\delta, \eps}^{-2} , 
        \\
        &\Big( a - \delta t_{\delta,\eps}^{-2}, \, \lambda(t_{\delta, \eps}) D_{\bx} \Phi_{\eps}(\bx_{\delta,\eps},\by_{\delta,\eps}) + \delta \bx_{\delta, \eps}, \, \bm{A} \Big) \in \ov{J^{2,+}} V^N(t_{\delta,\eps}, \bx_{\delta, \eps}), 
        \\
        & \Big( b, \, -\lambda(t_{\delta, \eps}) D_{\by_{\delta, \eps}} \Phi_{\eps}(\bx_{\delta, \eps},\by_{\delta, \eps}), \, \bm{B} \Big) \in \ov{J^{2,-}} V^N(s_{\delta,\eps}, \by_{\delta, \eps}), 
        \\
        & \begin{pmatrix}
            \bm{A} & \bm{0} 
            \\
            \bm{0} & -\bm{B}
        \end{pmatrix} \leq \bm{C} + \bm{C}^2, \quad \bm{C} \coloneqq D^2_{\bx,\by} \Phi_{\eps}(\bx_{\delta,\eps},\by_{\delta, \eps}).
    \end{align*}
    Using the equation for $V^N$ and the form of $D_{\bx} \Phi$ and $D_{\by} \Phi$, we find
    \begin{align*}
        &- a - \sum_{i,j = 1}^N \text{tr}\big(A^{i,j} \big) 
       + \frac{1}{N} \sum_{i = 1}^N H\Big( x_{\delta,\eps}^i, \lambda(t_{\delta, \eps}) \frac{x_{\delta,\eps}^i - y_{\delta,\eps}^i}{(|x_{\delta,\eps}^i - y_{\delta,\eps}^i|^2 + \eps^2)^{1/2}} + \delta x_{\delta, \eps}^i, m_{\bx_{\delta,\eps}}^N \Big) \leq 0, 
    \end{align*}
    as well as 
    \begin{align*}
        - b - \sum_{i,j = 1}^N \text{tr}\big(B^{i,j} \big) + \frac{1}{N} \sum_{i = 1}^N H\Big( y_{\delta,\eps}^i,  \lambda(t_{\delta, \eps}) \frac{x_{\delta,\eps}^i - y_{\delta,\eps}^i}{(|x_{\delta,\eps}^i - y_{\delta,\eps}^i|^2 + \eps^2)^{1/2}}, m_{\by_{\delta,\eps}}^N \Big) \geq 0.
    \end{align*}
    Subtracting the two inequalities, and recalling that $\lambda'(t) = - C_0 \lambda(t)$, we find that 
    \begin{align*}
      C_0 \lambda(t_{\delta, \eps}) & \Psi_{\eps}(\bx_{\delta,\eps} - \by_{\delta,\eps}) \leq -a + b \leq  \sum_{i,j = 1}^N \tr\big( A^{i,j} - B^{i,j} \big)
      \\
      &- \frac{1}{N} \sum_{i = 1}^N H\Big( x_{\delta,\eps}^i,\lambda(t_{\delta, \eps})\frac{x_{\delta,\eps}^i - y_{\delta,\eps}^i}{(|x_{\delta,\eps}^i - y_{\delta,\eps}^i|^2 + \eps^2)^{1/2}} + \delta x_{\delta, \eps}^i, m_{\bx_{\delta,\eps}}^N \Big)
        \\
        &\qquad \qquad +  \frac{1}{N} \sum_{i = 1}^N H\Big( y_{\delta,\eps}^i, \lambda(t_{\delta, \eps})  \frac{x_{\delta,\eps}^i - y_{\delta,\eps}^i}{(|x_{\delta,\eps}^i - y_{\delta,\eps}^i|^2 + \eps^2)^{1/2}}, m_{\by_{\delta,\eps}}^N \Big)
        \\
        &\leq \sum_{i,j = 1}^N \tr\big( A^{i,j} - B^{i,j} \big) 
        \\
        &\qquad + \frac{C_H}{N} \sum_{i = 1}^N \Big( \lambda(t_{\delta, \eps}) \frac{|x_{\delta,\eps}^i - y_{\delta,\eps}^i|}{(|x_{\delta,\eps}^i - y_{\delta,\eps}^i|^2 + \eps^2)^{1/2}}  + \delta |x_{\delta, \eps}^i| \Big) 
         \times \Big( \delta |x_{\delta, \eps}^i| + |x_{\delta, \eps}^i - y_{\delta,\eps}^i| + \frac{1}{N} \sum_{j = 1}^N |x_{\delta, \eps}^j - y_{\delta, \eps}^j| \Big)
         \\
         &\leq \sum_{i,j = 1}^N \tr\big( A^{i,j} - B^{i,j} \big) + C(\eps,N)  \big(1 + |\bx_{\delta, \eps}|^2\big) \delta 
         \\
         &\qquad + \frac{C_H}{N} \sum_{i = 1}^N \Big( \lambda(t_{\delta, \eps}) \frac{|x_{\delta,\eps}^i - y_{\delta,\eps}^i|}{(|x_{\delta,\eps}^i - y_{\delta,\eps}^i|^2 + \eps^2)^{1/2}} \Big) 
         \times \Big( |x_{\delta, \eps}^i - y_{\delta,\eps}^i| + \frac{1}{N} \sum_{j = 1}^N |x_{\delta, \eps}^j - y_{\delta, \eps}^j| \Big)
         \\
         &\leq \sum_{i,j = 1}^N \tr\big( A^{i,j} - B^{i,j} \big) + C(N)  \big(1 + |\bx_{\delta, \eps}|^2\big) \delta  + \frac{2 C_H \lambda(t)}{N} \sum_{i = 1}^N \frac{|x_{\delta,\eps}^i - y_{\delta,\eps}^i|^2}{(|x_{\delta,\eps}^i - y_{\delta,\eps}^i|^2 + \eps^2)^{1/2}} 
         \\
         &\leq  \sum_{i,j = 1}^N \tr\big( A^{i,j} - B^{i,j} \big) + C(N)  \big(1 + |\bx_{\delta, \eps}|^2\big) \delta  + 2 C_H \lambda(t) \Psi_{\eps}(\bx - \by). 
    \end{align*}
    Now, let $\bm{D} = (\bm{D}^{i,j})_{i,j = 1,\ldots,N} \in (\R^{d \times d})^{N \times N} \cong \R^{Nd \times Nd}$ denote the matrix $\bm{D}^{i,j} = I_{d \times d}$, and notice that 
    \begin{align*}
        \sum_{i,j = 1}^N \tr(A^{i,j}) = \tr\big( \bm{A} \bm{D} \big).
    \end{align*}
    We deduce that
    \begin{align*}
       \sum_{i,j = 1}^N \tr(A^{i,j} - B^{i,j}) = \tr\big( (\bm{A} - \bm{B}) \bm{D} \big) = \tr \bigg( \begin{pmatrix}
           \bm{A} & 0 \\
           0 & -\bm{B}
       \end{pmatrix} \begin{pmatrix}
           \bm{D} & \bm{D}\\
           \bm{D} & \bm{D}
       \end{pmatrix} \bigg) \leq \tr \bigg( \big(\bm{C} + \bm{C}^2 \big)  \begin{pmatrix}
           \bm{D} & \bm{D}\\
           \bm{D} & \bm{D}
       \end{pmatrix}  \bigg) \leq 0, 
    \end{align*}
    with the last line coming from the fact that $\bm{C}$ annihilates vectors of the form $(\bm{v}, \bm{v}) \in (\R^d)^2$.
    Thus we have 
    \begin{align*}
        3C_H \lambda(t_{\delta, \eps}) \Psi_{\eps}(\bx_{\delta,\eps} - \by_{\delta,\eps}) &\leq C_0 \lambda(t_{\delta, \eps}) \Psi_{\eps}(\bx_{\delta,\eps} - \by_{\delta,\eps}) \leq C(N)\big(\delta + \delta |\bx_{\delta, \eps}|^2 \big) + 2C_H \lambda(t) \Psi_{\eps}(\bx_{\delta,\eps} - \by_{\delta,\eps}).
    \end{align*}
    We obtain a contradiction by choosing $\delta$ small enough, and noting that for fixed $\eps$, $\delta |\bx_{\delta, \eps}|^2 \to 0$ as $\delta \downarrow{0}$. We deduce that we must $M_0 = 0$, which establishes \eqref{spatiallip}. Because $V^N$ is symmetric (i.e. $V^N(t,x^1,\ldots,x^N) = V^N(t,x^{\sigma(1)},\ldots,x^{\sigma(N)})$ for any partition $\sigma$ of $\{1,\ldots,N\}$), and 
    \begin{align*}
        \bd_1(m_{\bx}^N,m_{\by}^N) = \inf_{\sigma} \frac{1}{N} \sum_{i = 1}^N |x^i - y^{\sigma(i)}|, 
    \end{align*}
    it follows that 
    \begin{align*}
        |V^N(t,\bx) - V^N(t,\by)| \leq C \bd_1(m_{\bx}^N,m_{\by}^N). 
    \end{align*}
    It remains only to check the time regularity. Since the equation is invariant by time translation, it suffices to check the regularity near the terminal time. Indeed, for $h \in (0,T)$, $V^{N,h} := V^N( \cdot-h,\cdot)$ is a viscosity solution over $[h,T] \times (\R^d)^N$ of the same equation as $V^N$ except that the terminal condition is now $V^N(T-h, \cdot)$ instead of $G(m^N_{\bx})$. As a consequence, by the maximum principle, 
$$ \sup_{(t,\bx) \in [h,T] \times (\R^d)^N} | V^N(t-h,\bx) - V^N(t,\bx) | = \sup_{ \bx \in (\R^d)^N} |V^N(T-h, \bx) - G(m^N_{\bx})|. $$ 
Now, we observe that thanks to the space Lipschitz estimate obtained for $V^N$, there is $C>0$ independent from $N$ such that $V^N$ is a viscosity sub-solution of
\[
    - \partial_t V^{+,N} - \kappa \sum_{i,j=1}^N \tr(D_{x^{i}x^{j}}V^{+,N}) = C \quad \mbox{ in } [0,T) \times \R^d, \quad V^{+,N}(\bx)  = G(m^N_{\bx}) \quad \mbox{ in } (\R^d)^N
\]
and a viscosity super-solution of
\[
    - \partial_t V^{-,N} - \kappa \sum_{i,j=1}^N \tr(D_{x^{i}x^{j}}V^{-,N}) = -C \quad \mbox{ in } [0,T) \times \R^d, \quad V^{-,N}(\bx)  = G(m^N_{\bx}) \quad \mbox{ in } (\R^d)^N.
\]
Denoting respectively by $V^{+,N}$ and $V^{-,N}$ the (classical) solutions to the two equations above, we deduce from the comparison principle that $V^{N,-} \leq V^N \leq V^{N,+}$ over $[0,T] \times (\R^d)^N$. The solutions $V^{N,\pm}$ are given explicitly by  
\[
    V^{N,\pm} (t,\bx) = \pm C(T-t) + \E \bigl[ G(m^N_{\bX_t} )   \bigr] \quad X_t^{i} = x^{i} + \sqrt{2 \kappa} (W_T - W_t) \quad 1 \leq i \leq N
\]
for some (common) Brownian motion $(W_t)_{t \geq 0}$, and therefore, for all $(t,\bx) \in [0,T] \times (\R^d)^N$, 
\[
    |V^{N,\pm}(t,\bx) - G(m^N_{\bx})| \leq C(T-t) + \text{Lip}(G;d_1) \sqrt{2 \kappa} \sqrt{T-t}
\]
and thus
\[
    |V^N(t,\bx) - G(m^N_{\bx})| \leq 2C(T-t) + 2\text{Lip}(G;\bd_1) \sqrt{2 \kappa} \sqrt{T-t}.
\]
\end{proof}

\bibliographystyle{acm}
\bibliography{gms_rates}
\end{document}